\documentclass[11pt,a4paper]{amsart}
\usepackage{amssymb,amsmath,epsfig,graphics,mathrsfs}

\usepackage{fancyhdr}
\usepackage{hyperref}
\hypersetup{
 colorlinks   = true,
 urlcolor     = blue,
 linkcolor    = blue,
 citecolor   = red ,
 bookmarksopen=true
}

\usepackage{amsmath}
\usepackage{amsfonts}
\usepackage{amssymb}
\usepackage{amsthm}
\usepackage{epsfig,graphics,mathrsfs}
\usepackage{graphicx}
\usepackage{dsfont}

\usepackage[usenames, dvipsnames]{color} 

\usepackage{hyperref}

\newcommand*\MSC[1][1991]{\par\leavevmode\hbox{%
\textit{#1 Mathematical subject classification:\ }}}
\newcommand\blfootnote[1]{%
  \begingroup
  \renewcommand\thefootnote{}\footnote{#1}%
  \addtocounter{footnote}{-1}%
  \endgroup
}

\def \de {\partial}

\def \N {\mathbb{N}}

\def \phi {\varphi}
\def \RNu {\mathbb{R}^{N+1}}
\def \RN {\mathbb{R}^N}
\def \R {\mathbb{R}}

\def \K {\mathscr{K}}

\def \G{\Gamma}

\def \So {\mathscr{S}}
\newcommand{\As}{(-\mathscr A)^s}
\newcommand{\sA}{\mathscr A}

\newcommand{\Bps}{\mathfrak B^\sA_{s,p}}

\newcommand{\Rn}{\mathbb R^n}

\newcommand{\p}{\partial}

\newcommand{\la}{\lambda}

\numberwithin{equation}{section}

\newcommand{\beq}{\begin{equation}}
\newcommand{\bea}[1]{\begin{array}{#1} }
\newcommand{\eeq}{ \end{equation}}
\newcommand{\ea}{ \end{array}}

\newcommand{\ve}{\varepsilon}

\newcommand{\In}{\mathbf 1_E}
\newcommand{\Lp}{L^p}

\newcommand{\tr}{\operatorname{tr} B}

\newtheorem{theorem}{Theorem}[section]
\newtheorem{lemma}[theorem]{Lemma}
\newtheorem{proposition}[theorem]{Proposition}
\newtheorem{corollary}[theorem]{Corollary}

\numberwithin{equation}{section}

\begin{document}

\title[On the limiting behaviour, etc.]{On the limiting behaviour of some nonlocal seminorms:\\ a new phenomenon}

\blfootnote{\MSC[2020]{46E30, 35H10, 47B99}}

\date{}

\begin{abstract}
In this note we study the behaviour as $s\to 0^+$ of some semigroup based Besov seminorms associated with a non-symmetric and hypoelliptic diffusion with a drift. Our results generalise a previous one of Maz'ya and Shaposhnikova for the classical fractional Sobolev spaces $W^{s,p}$, and they also underscore a new phenomenon caused by the presence of the drift.
\end{abstract}

\author{Federico Buseghin}
\address{Department of Mathematical Sciences\\ University of Bath\\ Bath BA2 7AY, United Kingdom}

\vskip 0.2in

\email{fb588@bath.ac.uk}

\author{Nicola Garofalo}

\address{Dipartimento d'Ingegneria Civile e Ambientale (DICEA)\\ Universit\`a di Padova\\ Via Marzolo, 9 - 35131 Padova,  Italy}
\vskip 0.2in
\email{nicola.garofalo@unipd.it}

\thanks{The second author is supported in part by a Progetto SID (Investimento Strategico di Dipartimento) ``Non-local operators in geometry and in free boundary problems, and their connection with the applied sciences", University of Padova, 2017. The second and third authors are supported in part by a Progetto SID: ``Non-local Sobolev and isoperimetric inequalities", University of Padova, 2019.}

\author{Giulio Tralli}
\address{Dipartimento d'Ingegneria Civile e Ambientale (DICEA)\\ Universit\`a di Padova\\ Via Marzolo, 9 - 35131 Padova,  Italy}
\vskip 0.2in
\email{giulio.tralli@unipd.it}

\maketitle

\tableofcontents

\section{Introduction}\label{S:intro}

The limiting behaviour of some classical nonlocal seminorms has been the subject of increasing interest in recent years because of its connection with various function spaces, such as $\Lp$, Sobolev or $BV$ spaces. For $1\le p < \infty$ and $s\in(0,1)$ we denote by $W^{s,p}$ the Banach space of functions $f\in \Lp$ with finite Aronszajn-Gagliardo-Slobedetzky seminorm, 
\begin{equation}\label{ags}
[f]_{s,p} = \left(\int_{\RN} \int_{\RN} \frac{|f(x) - f(y)|^p}{|x-y|^{N+ps}} dx dy\right)^{1/p},
\end{equation}
see e.g. \cite{Ad} or also \cite{DPV}. In 
their celebrated works \cite{BBM1}, \cite{BBM2} (see also \cite{B}) Bourgain, Brezis and Mironescu discovered a new characterisation of the spaces $W^{1,1}$ and $BV$ based on the study of the limiting behaviour of the spaces $W^{s,p}$ as $s\to 1$. 
We also mention the earlier work \cite{MN}, in which the authors had already settled the case $p=2$ of the Bourgain-Brezis-Mironescu limiting theorem, and the work \cite{Ma}, which further analysed the case $p=1$. In their paper \cite{MS}  Maz'ya \& Shaposhnikova extended and simplified the results in \cite{BBM2}, and they also analysed the limit as $s\to 0^+$ of the seminorms \eqref{ags}. Regarding the latter, \cite[Theor. 3]{MS} states that if $f\in W^{s_0,p}$ for some $0<s_0<1$, then
\begin{equation}\label{MS}
\underset{s\to 0^+}{\lim} s\ [f]^p_{s,p} = \frac 2p \sigma_{N-1} ||f||^p_{L^p},
\end{equation}
where $\sigma_{N-1}$ is the measure of the unit sphere in $\RN$. 
These results have been extended and completed by several authors. For instance, one should see Milman \cite{Mi}, who placed them in the framework of interpolation spaces, Karadzhov, Milman and Xiao \cite{KMX}, Kolyada and Lerner \cite{KL}, Triebel \cite{Tr}, who generalized them in the context of Besov spaces, and Arcang\'eli and Torrens \cite{AT}.

To introduce the results in the present paper, we now make the key observation that theorem \eqref{MS} admits a dimension-free formulation using the heat semigroup $P^\Delta_t f(x) = e^{-t \Delta} f(x) = (4\pi t)^{-N/2} \int_{\Rn} e^{-\frac{|x-y|^2}{4t}} f(y) dy$. For $s>0$ and $1\le p<\infty$, consider the following heat Besov seminorm
\begin{equation}\label{ndelta}
\mathscr N^\Delta_{s,p}(f) = \left(\int_0^\infty  \frac{1}{t^{\frac{s p}2 +1}} \int_{\RN} P^\Delta_t\left(|f - f(x)|^p\right)(x) dx dt\right)^{\frac 1p}.
\end{equation}
We leave it as an easy exercise for the reader to recognise that 
\begin{equation}\label{equiv}
\mathscr N^\Delta_{s,p}(f)^p = \frac{2^{sp} \G(\frac{N+sp}2)}{\pi^{\frac N2}}\ [f]_{s,p}^p,
\end{equation}
where for $x>0$ we have denoted by $\G(x) = \int_0^\infty t^{x-1} e^{-t} dt$ the Euler gamma function.
Combining \eqref{equiv} with \eqref{MS}, we see that the theorem of Maz'ya \& Shaposhnikova  can be reformulated in terms of the heat seminorm \eqref{ndelta} in the following suggestive dimension-free fashion: assume that $f\in \underset{0<s<1}{\bigcup} W^{s,p}$, then
\begin{equation}\label{MSheat}
\underset{s\to 0^+}{\lim} s\ \mathscr N^\Delta_{s,p}(f)^p
= \frac 4p\ ||f||_{L^p}^p.
\end{equation}

The present work stems from the initial desire of understanding what happens to \eqref{MSheat} when the seminorm $\mathscr N^\Delta_{s,p}(f)$ is replaced by $\mathscr N^\sA_{s,p}(f)$, where $\sA$ is 
the infinitesimal generator of a wide class of non-symmetric semigroups with drift introduced by H\"ormander in his celebrated hypoellipticity paper \cite{Ho}. In the course of our study we have encountered a new,  unexpected phenomenon: the value of the corresponding limit in \eqref{MSheat} depends on the trace of the drift in $\sA$. But in order to state our results precisely, we need to introduce the relevant framework. 

Consider the Kolmogorov-Fokker-Planck operators in $\RNu$ defined as follows:
\begin{equation}\label{K0}
\mathscr K u = \mathscr A u  - \de_t u \overset{def}{=} \operatorname{tr}(Q \nabla^2 u) + <BX,\nabla u> - \de_t u = 0,
\end{equation}
where the $N\times N$ matrices $Q$ and $B$ have real, constant coefficients, and $Q = Q^\star \ge 0$. 
The operators $\K$ and $\sA$ in \eqref{K0} where introduced in \cite{Ho}, where H\"ormander showed that they are hypoelliptic if and only if the covariance matrix
\begin{equation}\label{Kt}
K(t) = \frac 1t \int_0^t e^{sB} Q e^{s B^\star} ds
\end{equation}
is invertible for every $t>0$. This condition will be henceforth tacitly assumed throughout this paper. Since one obviously has $K(t)\ge 0$, the invertibility of such matrix is equivalent to saying $K(t)>0$ for every $t>0$. Although in this paper we are mostly interested in the genuinely degenerate setting $N\ge 2$, our results are in fact true for any $N\ge 1$. With this assumption in place, we will routinely indicate with $X$ the generic point in $\R^N$, with $(X,t)$ the one in $\RNu$. 

Equations such as \eqref{K0} are of considerable interest in physics, probability and finance, and have been the subject of intense study during the past three decades. First, they obviously contain the classical heat equation, which corresponds to the non-degenerate model $Q = I_N$, $B = O_N$. More importantly, they encompass the Ornstein-Uhlenbeck operator (see \cite{OU}), which is obtained by taking $Q = I_N$ and $B = - I_N$ in \eqref{K0}, as well as the degenerate operator of Kolmogorov in $\R^{2n+1}$
\begin{equation}\label{kolmo}
\K_0 u = \Delta_v u + <v,\nabla_x u > - \p_t u,
\end{equation}
corresponding to the choice $N=2n$, 
$Q = \begin{pmatrix} I_n & 0_n\\ 0_n& 0_n\end{pmatrix}$, and $B = \begin{pmatrix} 0_n & 0_n\\ I_n & 0_n\end{pmatrix}$. Such operator arises in the kinetic theory of gases and was first introduced in the seminal note \cite{Ko} on Brownian motion. One should note that $\K_0$ fails to be parabolic since it is missing the diffusive term $\Delta_x u$. However, it does satisfy H\"ormander's hypoellipticity condition since one easily checks that $K(t)=\begin{pmatrix} I_n & t/2\ I_n\\ t/2\ I_n& t^2/3\ I_n\end{pmatrix}>0$ for every $t>0$. In this respect, it should be noted that Kolmogorov himself had already shown the hypoellipticity of his operator since in \cite{Ko} he constructed an explicit fundamental solution for $\K_0$ which is $C^\infty$ outside the diagonal.

Kolmogorov's construction was generalised in \cite{Ho}, where it was shown that, given $f\in \So$, the Cauchy problem $\K u = 0$ in $\RN\times (0,\infty)$, $u(X,0) = f(X)$ admits the unique solution $u(X,t) = \int_{\RN} p(X,Y,t) f(Y) dY$, where 
\begin{equation}\label{PtKt}
p(X,Y,t) = \frac{c_N}{V(t)} \exp\left( - \frac{m_t(X,Y)^2}{4t}\right).
\end{equation}
In \eqref{PtKt}, for $X, Y\in \RN$ we have let
\begin{align}\label{m}
m_t(X,Y) & = \sqrt{<K(t)^{-1}(Y-e^{tB} X ),Y-e^{tB} X >},\ \ \ \ \ \ \ t>0,
\end{align}
whereas, with $B_t(X,r) = \{Y\in \RN\mid m_t(X,Y)<r\}$ and $c_N=\omega_N (4\pi)^{-\frac N2}$, the notation $V(t)$ denotes the so-called volume function
\begin{equation}\label{VS}
V(t) = \operatorname{Vol}_N(B_t(X,\sqrt t)) = \omega_N  (\det(t K(t)))^{1/2},
\end{equation}
see \cite{GThls}. If we indicate with 
\begin{equation}\label{hsg}
P_t^{\sA} f(X) = \int_{\RN} p(X,Y,t) f(Y) dY
\end{equation}
 the H\"ormander semigroup, then it is well-known that, under the assumption that the matrix $B$ of the drift satisfies 
\begin{equation}\label{trace}
\operatorname{tr} B \ge 0,
\end{equation} 
we obtain a non-symmetric semigroup which is contractive on $L^p$, $1\le p\le \infty$. Because of the drift, such semigroup presents several new challenges with respect to the Riemannian or even sub-Riemannian setting. This is already apparent in H\"ormander's formula \eqref{m} above, in which the space and the time variables appear inextricably mixed.  
In a series of papers, see \cite{GT}, \cite{GThls}, \cite{GTfi} and \cite{GTiso}, two of us have recently developed, under the condition \eqref{trace}, some basic functional analytic aspects of the class \eqref{K0}. We note that Kolmogorov's operator \eqref{kolmo} satisfies \eqref{trace} since for such example we have in fact $\tr = 0$. For other operators of interest in physics that satisfy \eqref{trace} we refer the reader to the table in \cite[Figure 1]{GThls}. 
 
We thus come to the question of interest in this paper. In the work \cite{GTfi} a class of Besov spaces naturally associated with the semigroup $P_t^{\sA}$ was introduced. Namely, for any $s>0$ and $1\le p<\infty$ we defined 
the Besov space $\mathfrak B^\sA_{s,p}$ as the collection of all functions $f\in L^p$ such that
\begin{equation}\label{sn}
\mathscr N^\sA_{s,p}(f) = \left(\int_0^\infty  \frac{1}{t^{\frac{s p}2 +1}} \int_{\RN} P^\sA_t\left(|f - f(X)|^p\right)(X) dX dt\right)^{\frac 1p} < \infty.
\end{equation}
Although one might think of $\Bps$ as a natural generalisation of the spaces introduced by Taibleson in \cite{T1}, \cite{T2} using the heat semigroup,  the deeper properties of these spaces are somewhat elusive. The cases $p=2$ and $p=1$ of \eqref{sn} have a special interest in connection with the semigroup based theory of nonlocal isoperimetric inequalities developed in \cite{GTiso}. 

In the present paper we generalise the theorem of Maz'ya \& Shaposhnikova \eqref{MSheat} to the Besov spaces $\mathfrak B_{s,p}^\sA$. Our main result in this direction is the following. 

\begin{theorem}\label{T:MSallp}
Let $1\leq p <\infty$, and assume \eqref{trace}. Suppose that $f\in \underset{0<s<1}{\bigcup}\mathfrak B^\sA_{s,p}$. Then,
\begin{equation}\label{azza}
\underset{s\to 0^+}{\lim} s \mathscr N^\sA_{s,p}(f)^p = \begin{cases}
\frac{4}{p} ||f||_p^p,\ \ \ \ \ \ \ \mbox{if }  \tr = 0,
\\
\\
\frac{2}{p}||f||_p^p,\ \ \ \ \ \ \ \mbox{if } \tr >0.
\end{cases}
\end{equation}
\end{theorem}

The reader should note the unexpected discrepancy between the cases $\tr = 0$ and $\tr > 0$ in \eqref{azza} above. For instance, whereas for the Besov space generated by Kolmogorov operator \eqref{kolmo} the limit in \eqref{azza} equals $\frac{4}{p} ||f||_p^p$, for the Kolmogorov operator with friction in $\R^{2n+1}$, 
\[
\K_1 u = \Delta_v u  + <v,\nabla_v u> + <v,\nabla_x u> - \de_t u,
\]
 for which $\tr = n>0$, the analogous limit equals $\frac{2}{p} ||f||_p^p$! We also make the remark that, if we agree to say that a measurable set $E\subset \R^N$ has finite $s$-\emph{perimeter} if $\In\in \mathfrak B^\sA_{2s,1}$ and to define the $s$-perimeter associated to $\sA$ as
\[
\mathfrak P_{\sA,s}(E) \overset{def}{=} \mathscr N^\sA_{2s,1}(\In),
\]
then we can readily derive from Theorem \ref{T:MSallp} the following asymptotic result: suppose that $E$ has finite $\sigma$-perimeter for some $\sigma \in (0,\frac{1}{2})$, then
$$
\underset{s\to 0^+}{\lim} s\, \mathfrak P_{\sA,s}(E) = \begin{cases}
2 |E|,\ \ \ \ \ \ \ \mbox{if }  \tr = 0,
\\
\\
\ |E|,\ \ \ \ \ \ \ \,\, \mbox{if } \tr >0.
\end{cases}$$
In \cite{GTbbmd} two of us have addressed the limiting behaviour of $s$-perimeters as $s\nearrow 1/2$ in a different but related sub-Riemannian setting.

Having stated our main result, we now briefly describe the organisation of the present paper. 
In Section \ref{S:prel} we analyse the behaviour of the volume function $V(t)$ defined by \eqref{VS}, and of the H\"ormander semigroup $P^\sA_t$ in \eqref{hsg}. Our key results are Proposition \ref{P:boom} and Proposition \ref{P:balapoint}. The former complements and completes Proposition \ref{P:boom0} below, which was proved in \cite{GThls}. The latter establishes the limiting pointwise behaviour of the
fractional powers
\begin{equation}\label{defpower}
(-\sA)^s f(X) = - \frac{s}{\Gamma(1-s)} \int_0^\infty \frac{1}{t^{1+s}} \left(P^\sA_t f(X) -  f(X)\right) dt, 
\end{equation}
in dependence of the eigenvalues of the drift matrix $B$ in \eqref{K0}.
In Section \ref{S:prop} we gather some basic properties of the Besov spaces $\mathfrak B^\sA_{s,p}$ under the assumption \eqref{trace}. The main result is Proposition \ref{P:density}, which establishes a key density property for such spaces. This result generalises the well-known one for the classical spaces $W^{s,p}$, see e.g. \cite[Theor. 7.38]{Ad} and plays a key role in the present work. Section \ref{S:lim} is devoted to proving Theorem \ref{T:MSallp}. Such proof is based on the four Lemmas \ref{L:uno}-\ref{L:quattro}. Finally, in Section \ref{S:fracpow}  we analyse the asymptotic behaviour as $s\to 0+$ of the fractional powers \eqref{defpower}
under the hypothesis $f\in \underset{0<s<1}{\bigcup}\mathfrak B^\sA_{s,p}$. We note that this assumption is the same as in Theorem \ref{T:MSallp}.
The main results are Theorem \ref{T:2} and Proposition \ref{P:balap1not}, whose proofs are based on some results of independent interest that are closely connected to the arguments of Section \ref{S:lim}. The reader is referred to \cite{GT} for the calculus of the nonlocal operators \eqref{defpower}, and to \cite{GThls}, \cite{GTiso} for optimal Sobolev type embeddings and isoperimetric inequalities.


\section{On the volume function $V(t)$ and the semigroup $P_t^\sA$}\label{S:prel}

We start by collecting some preliminary material that will be used throughout the paper. For more extensive information we refer the reader to \cite[Sec. 2]{GT}, \cite[Sec. 2]{GThls} and \cite{GTfi}. Generic points in $\RN$ will be denoted with the letters $X, Y$ and their Euclidean norms with $|X|$, $|Y|$. The trace and the determinant of a matrix $M$ will be indicated with $\operatorname{tr} M$ and $\operatorname{det} M$ respectively, $M^\star$ denotes the transpose of $M$, and we let $\|M\| = \sup_{|X|=1} |MX|$. Given a measurable set $E\subset \RN$, we also denote by $|E|$ its $N$-dimensional Lebesgue measure. All the function spaces in this paper are based on $\RN$, thus we will routinely avoid reference to the ambient space. For instance, the Schwartz space of rapidly decreasing functions in $\RN$ will be denoted by $\So$, and for $1\leq p \leq \infty$ we let $L^p = L^p(\RN)$. The norm in $\Lp$ will be denoted by $||\cdot||_p$, instead of $||\cdot||_{\Lp}$. Moreover, to simplify the notation we will henceforth indicate with $P_t$, instead of $P_t^\sA$, the H\"ormander semigroup \eqref{hsg} associated with \eqref{K0}, and use the notation $P^\star_t$ for its adjoint. These semigroups possess the following two basic properties:
\begin{equation}\label{stochcompl}
P_t 1=1,\ \ \mbox{i.e.}\ \int_{\R^N}p(X,Y,t)dY= 1,\ \ \ \ \ \ \ \ \ \ \ \ \ \ \ \ \ \ \ \ X\in\R^N, t>0;
\end{equation}
\begin{equation}\label{aggstoch}
P^*_t 1=e^{-t\tr},\ \ \mbox{ i.e. }\ \int_{\R^N}p(X,Y,t)dX= e^{-t\tr}, \ \ \ \ Y\in\R^N, t>0.
\end{equation}
From \eqref{stochcompl} and \eqref{aggstoch} one easily recognises that $||P_t f||_p\leq ||f||_p$ when $\tr\geq 0$. More in general, we have the following $L^p\to L^q$ ultracontractivity of the semigroup $\{P_t\}_{t>0}$, see \cite[Prop. 2.3]{GTfi}. Hereafter, the notation $V(t)$ will indicate the volume function introduced in \eqref{VS}.

\begin{proposition}\label{P:ptq}
For every $1\leq q< \infty$ and $p\geq q$,  we have $P_t: L^q\to L^p$ for any $t>0$, with
\begin{equation}\label{tuttenorme}
||P_t f||_{p} \le \frac{C}{V(t)^{\frac{1}{q}-\frac{1}{p}}}e^{- t \frac{\operatorname{tr} B}{p}} ||f||_{q},
\end{equation}
for some constant $C = C(N,q,p)>0$.
\end{proposition}

It is clear that in order to be able to effectively exploit Proposition \ref{P:ptq} it is critical to know the large time behaviour of the volume function $V(t)$. In this respect, we recall the following result, which is \cite[Proposition 3.1]{GThls}.
The notation $\sigma(B)$ indicates the spectrum of the drift matrix $B$ in \eqref{K0}. 

\begin{proposition}\label{P:boom0}
Suppose that $N\geq 2$ and \eqref{trace} hold. Then:
\begin{itemize}
\item[(i)] there exists a constant $c_1>0$ such that $V(t)\geq c_1 t$ for all $t\geq 1$;
\item[(ii)] moreover, if $\max\{\Re(\lambda)\mid \lambda\in \sigma(B)\}=L_0>0$, there exists a constant $c_0$ such that $V(t)\geq c_0 e^{L_0 t}$ for all $t\geq 1.$
\end{itemize}
\end{proposition}

For the purpose of this paper, we will need the following improvement of Proposition \ref{P:boom0} which is valid for any $N\geq 1$ and also encompasses the case $\tr<0$.

\begin{proposition}\label{P:boom}
If $\max\{\Re(\lambda)\mid \lambda\in \sigma(B)\}\geq 0$, then there exists a constant $c_0>0$ such that 
\begin{equation}\label{sqrtbound}
V(t)\geq c_0 \sqrt{t} \quad\mbox{for all }t\geq 1.
\end{equation}
If instead $\max\{\Re(\lambda)\mid \lambda\in \sigma(B)\}< 0$, then as $t\nearrow \infty$ we have
\begin{equation}\label{measure}
tK(t)\nearrow K_\infty\overset{def}{=}\int_0^\infty e^{sB} Q e^{s B^\star} ds <\infty.
\end{equation}
\end{proposition}
\begin{proof}
Without loss of generality we can assume, up to a change of variables in $\R^N$, that the matrix $B^*$ is in the following block-diagonal real Jordan canonical form
$$B^\star=\begin{pmatrix}
J_{n_1}(\lambda_1) &  &  &  &  &  & 
\\
 & \ddots &  &  & 0 &   & 
\\
 &  & J_{n_q}(\lambda_q) &  &  &  & \\
 &  &  & C_{m_1}(a_1,b_1) &  &  & 
\\
 &  & 0 &  &  & \ddots & 
\\
 &  &  &  &  &   & C_{m_p}(a_p,b_p)
\end{pmatrix},$$
where $\sigma(B)=\sigma(B^\star)=\{\lambda_1,\ldots,\lambda_q,a_1\pm ib_1,\ldots, a_p\pm ib_p\}$ with $\lambda_k, a_\ell, b_\ell\in \R$ ($b_\ell\neq 0$), $n_1+\ldots+n_q+2m_1+\ldots+2m_p=N$ with $n_k, m_\ell\in \N$, and the $n_k\times n_k$ matrix $J_{n_k}(\lambda_k)$ and the $2m_\ell\times 2m_\ell$ matrix $C_{m_\ell}(a_\ell,b_\ell)$ are respectively in the form
$$J_{n_k}(\lambda_k)=\begin{pmatrix}
\lambda_k & 1 & 0 & \ldots & 0  \\
0 & \lambda_k & 1 & \ldots & 0  \\
0 & 0 & \ddots & \ddots & 0  \\
0 & \ldots & 0 & \lambda_k & 1  \\
0 & 0 & \ldots & 0 & \lambda_k 
\end{pmatrix},\quad C_{m_\ell}(a_\ell,b_\ell)=\begin{pmatrix}
a_\ell & -b_\ell & 1 & 0 & 0 & \ldots & \ldots & 0  \\
b_\ell & a_\ell & 0 & 1 & 0 & \ldots & \ldots & 0  \\
0 & 0 & a_\ell & -b_\ell & 1 & 0 & \ldots & 0  \\
0 & 0 & b_\ell & a_\ell & 0 & 1 & \ldots & 0 \\
\vdots & \vdots & 0 & 0 & \ddots & \ddots & 1 & 0 \\
\vdots & \vdots & 0 & 0 & \ddots & \ddots & 0 &  1\\
0 & 0 & \ldots & \ldots & 0 & 0 & a_\ell & -b_\ell \\
0 & 0 & \ldots & \ldots & 0 & 0 & b_\ell & a_\ell \\
\end{pmatrix}.
$$
In these notations, the two mutually exclusive possibilities $\max\{\Re(\lambda)\mid \lambda\in \sigma(B)\}\geq 0$ and $\max\{\Re(\lambda)\mid \lambda\in \sigma(B)\}< 0$ respectively correspond  to the the following conditions:
\begin{itemize}
\item[(a)] there is at least one $k_0\in\{1,\ldots,q\}$ such that $\lambda_{k_0}\geq 0$, or at least one $\ell_0\in\{1,\ldots,p\}$ such that $a_{\ell_0}\geq 0$;
\item[(b)] for every $k\in\{1,\ldots,q\}$ and $\ell\in\{1,\ldots,p\}$ we have $\lambda_{k},\,a_{\ell}<0$.
\end{itemize}
Suppose at first that case $(a)$ occurs. A thorough review of the proof of \cite[Proposition 3.1]{GThls} tells us that, regardless of the sign assumption on $\tr$, the following holds:
\begin{itemize}
\item[-] if there exists $\ell_0\in\{1,\ldots,p\}$ such that $a_{\ell_0}> 0$ then, for some $C_+>0$, we have $\operatorname{det}\left(tK(t)\right)\geq C_+ e^{2a_{\ell_0} t}$ for all $t\geq 1$;
\item[-] if there exists $k_0\in\{1,\ldots,q\}$ such that $\lambda_{k_0}> 0$ then, for some $C_+>0$, we have $\operatorname{det}\left(tK(t)\right)\geq C_+ e^{2\lambda_{k_0} t}$ for all $t\geq 1$;
\item[-] if there exists $\ell_0\in\{1,\ldots,p\}$ such that $a_{\ell_0}= 0$ then, for some $C_0>0$, we have $\operatorname{det}\left(tK(t)\right)\geq C_0 t^2$ for all $t\geq 1$;
\item[-] if there exists $k_0\in\{1,\ldots,q\}$ with $n_{k_0}\geq 2$ such that $\lambda_{k_0}= 0$ then, for some $C_0>0$, we have $\operatorname{det}\left(tK(t)\right)\geq C_0 t^3$ for all $t\geq 1$.
\end{itemize}
Being in case $(a)$, the only possibility which is left out from the analysis of the previous list is the following:
\begin{equation}\label{leftout}
\mbox{suppose there exists }k_0\in\{1,\ldots,q\}\mbox{ with }n_{k_0}=1\mbox{ such that }\lambda_{k_0}= 0.
\end{equation}
Under assumption \eqref{leftout}, we know there exists a vector $v_0\in\RN$, with $|v_0| = 1$, which is in the kernel of $B^*$ (i.e., an eigenvector with eigenvalue $\lambda_{k_0}= 0$). From the H\"ormander condition (see \cite[Proposition 2.12]{GThls}) we deduce that $v_0\notin \operatorname{Ker} Q$, that is $\langle Q v_0,v_0\rangle >0$ holds true. Therefore, denoting by $\lambda_M(t)$ the largest eigenvalue of $tK(t)$, we obtain
$$\lambda_M(t) \geq \langle tK(t) v_0,v_0\rangle =\int_0^t \langle Q e^{sB^\star}v_0, e^{sB^\star}v_0\rangle ds = \int_0^t \langle Q v_0,v_0\rangle ds=t \langle Q v_0,v_0\rangle.$$
On the other hand, since $t\mapsto tK(t)$ is monotone increasing in the sense of matrices (recall \eqref{Kt}), for $t\geq 1$ all the eigenvalues of $tK(t)$ are larger than the minimum eigenvalue of $K(1)$ which is strictly positive by H\"ormander condition and can be denoted by $\lambda_1$: from this fact we infer that
$$
\operatorname{det}\left(tK(t)\right)\geq (\lambda_1)^{N-1} \lambda_M(t)\geq (\lambda_1)^{N-1}\left\langle Q v_0,v_0\right\rangle t.
$$
If we put together all the previous information concerning the lower bound for $\operatorname{det}\left(tK(t)\right)$, we conclude that in case $(a)$ we have
$$\operatorname{det}\left(tK(t)\right)\gtrsim t\quad\mbox{ for }t\geq 1.$$
By recalling the definition of $V(t)$ in \eqref{VS}, this implies the validity of \eqref{sqrtbound} for some constant $c_0>0$.

Suppose now that case $(b)$ occurs. The conclusion in \eqref{measure} is known, see e.g. \cite[Section 6]{DZ}. For the reader's convenience, we provide a quick proof adapted to our setting. Since $t\mapsto tK(t)$ is monotone increasing and positive definite, to establish \eqref{measure} it suffices to prove that $\left\langle tK(t) v,v\right\rangle$ is bounded above uniformly in $t$, for every unit vector $v\in\RN$. With this objective in mind, suppose that we knew that there exist constants $\alpha, C_B>0$  such that for all $t\ge 0$
\begin{equation}\label{vanishB}
\|e^{tB^*}\|\leq C_B e^{- \alpha t}.
\end{equation}
Then, denoting by $\Lambda_Q$ the largest eigenvalue of the matrix $Q$, for any $v$ with $|v|=1$ and for all $t$ we would have from \eqref{vanishB}
$$\left\langle tK(t) v,v\right\rangle \leq \Lambda_Q\int_0^t |e^{sB^*}v|^2 ds\leq \Lambda_Q\int_0^{\infty}\|e^{sB^*}\|^2 ds< \infty.$$
To complete the proof of part (b) we are thus left with showing \eqref{vanishB}. This estimate can be showed by verifying that
$$e^{t J_{n_k}(\lambda_k)}=e^{\lambda_k t}\begin{pmatrix}
1 & t & \frac{t^2}{2} & \ldots & \frac{t^{n_k-1}}{(n_k-1)!}  \\
0 & 1 & t & \ldots & \frac{t^{n_k-2}}{(n_k-2)!}  \\
0 & 0 & \ddots & \ddots & \vdots  \\
0 & \ldots & 0 & 1 & t  \\
0 & 0 & \ldots & 0 & 	1 
\end{pmatrix}$$
and
$$e^{t C_{m_\ell}(a_\ell,b_\ell)}=e^{a_\ell t}\begin{pmatrix}
R_{t b_\ell} & tR_{t b_\ell} & \frac{t^2}{2}R_{t b_\ell} & \ldots & \frac{t^{m_\ell-1}}{(m_\ell-1)!}R_{t b_\ell}  \\
0 & R_{t b_\ell} & tR_{t b_\ell} & \ldots & \frac{t^{m_\ell-2}}{(m_\ell-2)!}R_{t b_\ell}  \\
0 & 0 & \ddots & \ddots & \vdots  \\
0 & \ldots & 0 & R_{t b_\ell} & tR_{t b_\ell}  \\
0 & 0 & \ldots & 0 & 	R_{t b_\ell} 
\end{pmatrix},$$
where $R_{t b_\ell}= \begin{pmatrix}
\cos{(tb_\ell)} & -\sin{(tb_\ell)}   \\
\sin{(tb_\ell)} & \cos{(tb_\ell)}
\end{pmatrix}$.
Then, for any block $B^*_j$ of $B^*$ (either of type $J_{n_k}(\lambda_k)$ or  $C_{m_\ell}(a_\ell,b_\ell)$), one has
$$
\|e^{tB_j^*}\|\lesssim t^{d_j} e^{-L_j t}\quad \mbox{ for }t\geq 1,
$$
where $d_j\geq 0$ is a suitable power and $L_j$ is strictly positive (because all the $\lambda_k, a_\ell$ are strictly
negative). This implies the validity of \eqref{vanishB}.

\end{proof}

The expression in \eqref{PtKt} trivially implies an upper bound $|P_t f(X)|\leq \frac{c_N}{V(t)}$. Hence, if we assume $\max\{\Re(\lambda)\mid \lambda\in \sigma(B)\}\geq 0$, the rate of blowup for $V(t)$ that ensues from Proposition \ref{P:boom0} and \eqref{sqrtbound} of Proposition \ref{P:boom} provides us with a critical information on the rate of vanishing of the semigroup $P_t$ as $t\to \infty$. What is left out is the situation in which $\max\{\Re(\lambda)\mid \lambda\in \sigma(B)\}< 0$. In the next result we show that, in this case, $P_t f$ converges as $t\to \infty$ with an exponential rate to the average of $f$ with respect to the invariant Gaussian measure. 

\begin{proposition}\label{equilibrium}
Assume $\max\{\Re(\lambda)\mid \lambda\in \sigma(B)\}< 0$. Then, for every $f\in \So$ and $X\in \RN$, there exists a $C_{f,X}>0$ such that for all $t\geq 1$,
\begin{equation}\label{expoconv}
\left|P_t f(X) - \frac{(4\pi)^{-\frac N2}}{\sqrt{{\rm{det}}\ K_\infty}}\int_{\RN} f(Y)e^{-\frac{<K^{-1}_\infty Y,Y>}{4}}dY\right|\leq C_{f,X} e^{-\alpha t},
\end{equation}
where $\alpha>0$ is the constant in \eqref{vanishB}.
\end{proposition}

\begin{proof}
Take $f\in\So$ and denote
$$
m_\infty(f)=\frac{(4\pi)^{-\frac N2}}{\sqrt{{\rm{det}}\ K_\infty}}\int_{\RN} f(Y)e^{-\frac{<K^{-1}_\infty Y,Y>}{4}}dY.
$$
We first note that, for any $X\in\RN$,
\begin{equation}\label{continfty}
P_t f(X)\longrightarrow m_\infty(f)\quad\mbox{ as }t\to \infty.
\end{equation}
To prove \eqref{continfty} we observe that, as a consequence of \eqref{PtKt}, \eqref{m}, \eqref{measure} and \eqref{vanishB}, we have for any $X,Y\in\RN$, 
$$p(X,Y,t)=\frac{(4\pi)^{-\frac N2}}{\sqrt{{\rm{det}}\left(tK(t)\right)}}e^{-\frac{<K(t)^{-1}(Y-e^{tB} X ),Y-e^{tB} X >}{4t}}\ \underset{t\to \infty}{\longrightarrow}\ \frac{(4\pi)^{-\frac N2}}{\sqrt{{\rm{det}}\ K_\infty}}e^{-\frac{<K^{-1}_\infty Y,Y>}{4}}.$$
This limit, and Lebesgue dominated converge theorem, imply \eqref{continfty} once we observe that for $t\geq 1$ one has
$$|p(X,Y,t) f(Y)|\leq \frac{(4\pi)^{-\frac N2}}{\sqrt{{\rm{det}}\ K(1)}}|f(Y)|\in L^1.$$
Now, fix $X\in\RN$ and let $\alpha>0$ be the constant in \eqref{vanishB}. With $t(\rho) = \frac 1\alpha \log \frac 1\rho$, we now define a function $g_{X}:(0,1)\to \R$ by the formula 
\[
g_X(\rho) = \begin{cases}
m_\infty(f),\ \ \ \ \ \ \ \ \ \ \rho = 0,
\\
P_{t(\rho)} f(X),\ \ \ \ \ \ 0<\rho<1.
\end{cases}
\]
Thanks to \eqref{continfty} the function $g_{X}(\rho)$ is continuous up to $\rho=0$. Moreover, for $f\in\So$ the chain rule gives for any $\rho\in (0,1)$
$$g'_{X}(\rho)=-\frac{1}{\alpha \rho}\int_{\RN}\frac{\de p}{\de t}\left(X,Y,t(\rho)\right)f(Y) dY.$$
By the mean value theorem we thus find for all $t\geq 1$,
\begin{align*}
&\left|P_t f(X) - m_\infty(f)\right|=\left|g_{X}(e^{-\alpha t}) -g_{X}(0)\right|\leq e^{-\alpha t} \sup_{\rho\in (0, e^{-\alpha t})}\left|g'_{X}(\rho)\right|\\
&\leq \frac{e^{-\alpha t}}{\alpha} \sup_{\tau\geq 1} \int_{\RN}e^{\alpha \tau}\left|\frac{\de p}{\de \tau}(X,Y,\tau)\right||f(Y)| dY.
\end{align*}
To complete the proof of \eqref{expoconv} we will show that there exists $C>0$ (depending on $f$ and $X$) such that
\begin{equation}\label{bounderiv}
\sup_{\tau\geq 1} \int_{\RN}e^{\alpha \tau}\left|\frac{\de p}{\de \tau}(X,Y,\tau)\right||f(Y)| dY \le C.
\end{equation}
The identity $ \frac{d}{d\tau}(\tau K(\tau))= e^{\tau B} Q e^{\tau B^*}$ and a direct computation show that
\begin{align*}
\frac{\de p}{\de \tau}(X,Y,\tau)=p(X,Y,\tau)&\bigg(-\frac{1}{2}{\rm{tr}}\left(e^{\tau B} Q e^{\tau B^*} \left(\tau K(\tau)\right)^{-1}\right)
\\
&+ \frac{1}{4}<e^{\tau B} Q e^{\tau B^*}\left(\tau K(\tau)\right)^{-1}\left(Y-e^{\tau B}X\right),\left(\tau K(\tau)\right)^{-1}\left(Y-e^{\tau B}X\right)>
\\
& + \frac{1}{2}<\left(\tau K(\tau)\right)^{-1}\left(Y-e^{\tau B}X\right),e^{\tau B} BX>\bigg).
\end{align*}
We are going to estimate separately the three terms appearing in the right-hand side of the latter identity using the following facts: (a) the matrix inequality $\tau K(\tau)\geq K(1)>0$ for $\tau\geq 1$; (b) the fact that the largest eigenvalue of the nonnegative matrix $e^{\tau B} Q e^{\tau B^*}$ is smaller than $\Lambda_Q \|e^{\tau B^*}\|^2$ (where $\Lambda_Q$ denotes the largest eigenvalue of $Q$); and (c) the key exponential decay established in \eqref{vanishB}.
 We thus obtain for all $\tau\geq 1$,
$$
0<{\rm{tr}}\left( e^{\tau B} Q e^{\tau B^*} \left(\tau K(\tau)\right)^{-1}\right)\leq \Lambda_Q \|e^{\tau B^*}\|^2{\rm{tr}}\left(K^{-1}(1)\right)\leq C^2_B\Lambda_Q{\rm{tr}}\left(K^{-1}(1)\right) e^{-2\alpha\tau}.
$$
Secondly, for all $\tau\geq 1$ we have
\begin{align*}
&0\leq <e^{\tau B} Q e^{\tau B^*}\left(\tau K(\tau)\right)^{-1}\left(Y-e^{\tau B}X\right),\left(\tau K(\tau)\right)^{-1}\left(Y-e^{\tau B}X\right)>\\
&\leq \Lambda_Q \|e^{\tau B^*}\|^2 \left| \left(\tau K(\tau)\right)^{-1}\left(Y-e^{\tau B}X\right) \right|^2\leq \Lambda_Q \|e^{\tau B^*}\|^2\|K^{-1}(1)\|^2 \left| Y-e^{\tau B}X \right|^2\\
&\leq 2\Lambda_Q \|e^{\tau B^*}\|^2\|K^{-1}(1)\|^2 \left( |Y|^2 + \|e^{\tau B}\|^2|X|^2 \right)\\
&\leq 2\Lambda_Q C_B^2 \|K^{-1}(1)\|^2 \left( |Y|^2 + C_B^2|X|^2 \right) e^{-2\alpha\tau }.
\end{align*}
Finally, for $\tau\geq 1$ we bound the last term as follows
\begin{align*}
&\left|<\left(\tau K(\tau)\right)^{-1}\left(Y-e^{\tau B}X\right),e^{\tau B} BX>\right|\\
&\leq \|e^{\tau B^*}\|\left(\tau K(\tau)\right)^{-1}\left(Y-e^{\tau B}X\right) \left|\right| |BX|\leq \|e^{\tau B^*}\| \|K^{-1}(1)\|\left( |Y| + \|e^{\tau B}\| |X| \right) \|B\| |X|\\
&\leq \|K^{-1}(1)\|\|B\|( |Y| + C_B |X|)|X| e^{-\alpha\tau}.
\end{align*}
Inserting these three estimates in the above expression of $\frac{\de p}{\de \tau}(X,Y,\tau)$, we obtain for some $\bar{C}>0$ and all $\tau\geq 1$
\begin{align*}
&\left|\frac{\de p}{\de \tau}(X,Y,\tau)\right|\leq \frac{1}{2}p(X,Y,\tau)\left( C^2_B\Lambda_Q{\rm{tr}}\left(K^{-1}(1)\right) e^{-2\alpha\tau} +\right.\\
&\left. \Lambda_Q C_B^2 \|K^{-1}(1)\|^2 \left( |Y|^2 + C_B^2|X|^2 \right) e^{-2\alpha\tau } + \|K^{-1}(1)\|\|B\|( |Y| + C_B |X|)|X| e^{-\alpha\tau}\right)\\
&\leq \bar{C} \left( 1+  |Y|^2 + |X|^2\right)p(X,Y,\tau) e^{-\alpha\tau}.
\end{align*}
Using now \eqref{stochcompl} and the fact that $f\in\So$, we finally find for all $\tau\geq 1$,
\begin{align*}
&\int_{\RN}\left|\frac{\de p}{\de \tau}(X,Y,\tau)\right||f(Y)| dY \leq \bar{C} e^{-\alpha\tau}\int_{\RN} p(X,Y,\tau) \left( 1+  |Y|^2 + |X|^2\right) |f(Y)| dY\\
&\leq  e^{-\alpha\tau} \bar{C} \sup_{Y\in\RN}\left| \left( 1+  |Y|^2 + |X|^2\right) f(Y) \right|.
\end{align*}
This establishes \eqref{bounderiv} thus completing the proof of the lemma.

\end{proof}

Combining Propositions \ref{P:boom0}, \ref{P:boom} and \ref{equilibrium} with the case $p = \infty$ of Proposition \ref{P:ptq} we obtain a complete understanding of the pointwise behaviour of the semigroup $P_t f(X)$ as $t\to \infty$. It is interesting to notice how such behaviour depends in an essential way on the eigenvalues of the drift matrix $B$ in \eqref{K0}. In the same spirit, in the next result we analyse the pointwise limit as $s\to 0^+$ of the fractional powers \eqref{defpower}. In Section \ref{S:fracpow} this analysis will be complemented  by the study of the limiting behaviour in $L^p$ spaces of these nonlocal operators, under the assumption \eqref{trace}.

{\allowdisplaybreaks
\begin{proposition}\label{P:balapoint}
Let $f\in \mathscr S$ and $X\in\RN$. The following holds:
\begin{itemize}
\item[(i)] if $\max\{\Re(\lambda)\mid \lambda\in \sigma(B)\}\geq 0$, then one has
\[
\underset{s\to 0^+}{\lim}\ (-\sA)^s f(X) = f(X).
\]
\item[(ii)] if, on the other hand, $\max\{\Re(\lambda)\mid \lambda\in \sigma(B)\}< 0$, then one has
\[
\underset{s\to 0^+}{\lim}\ (-\sA)^s f(X) = f(X)-\frac{(4\pi)^{-\frac N2}}{\sqrt{{\rm{det}}\ K_\infty}}\int_{\RN} f(Y)e^{-\frac{<K^{-1}_\infty Y,Y>}{4}}dY.
\]
\end{itemize}
\end{proposition}
}

{\allowdisplaybreaks
\begin{proof}
To begin we recall that, for functions $f\in\So$, the definition of the fractional powers $\As f (X)$ in \eqref{defpower} makes a pointwise sense regardless of any sign assumption on the eigenvalues of $B$, see \cite[Section 3]{GT}. Suppose first that $\max\{\Re(\lambda)\mid \lambda\in \sigma(B)\}\geq 0$. We make use of the well-known identity 
\begin{equation}\label{identity}
\frac{s}{\Gamma(1-s)} \int_0^\infty \frac{1-e^{-t}}{t^{1+s}} dt = 1.
\end{equation}
From \eqref{defpower} and \eqref{identity} we find
\begin{align*}
& 
(-\sA)^s f(X) - f(X) = - \frac{s}{\Gamma(1-s)} \int_0^\infty \frac{1}{t^{1+s}} \left((P_t f(X) -  f(X)) + (1-e^{-t}) f(X)\right) dt
\\
& = - \frac{s}{\Gamma(1-s)} \int_0^1 \frac{1}{t^{1+s}} \left((P_t f(X) -  f(X)) + (1-e^{-t}) f(X)\right) dt
\\
& - \frac{s}{\Gamma(1-s)} \int_1^\infty \frac{1}{t^{1+s}} \left(P_t f(X) -e^{-t} f(X)\right) dt.
\end{align*}
At this point, it suffices to show that either one of the two integrals in the right-hand side of the latter identity converges to $0$ as $s\to 0^+$. Concerning the integral on $(0,1)$, we know from \cite[Lemma 2.5 (case $p=\infty$)]{GT} that $|P_t f(X) - f(X)| \le ||\sA f||_\infty\ t$ for $0\leq t\leq1$. Since also $|1-e^{-t}| \leq t$ for $t\in[0,1]$, we obtain
\begin{align*}
& \frac{s}{\Gamma(1-s)}\left|\int_0^1 \frac{1}{t^{1+s}} \left((P_t f(X) -  f(X)) + (1-e^{-t}) f(X)\right) dt\right| 
\\
& \leq \frac{s}{\Gamma(1-s)}(||\sA f||_\infty + ||f||_\infty) \int_0^1 \frac{dt}{t^{s}} = \frac{s}{(1-s)\Gamma(1-s)}(||\sA f||_\infty + ||f||_\infty)\ \underset{s\to 0^+}{\longrightarrow}\ 0.
\end{align*}
We now consider the integral on $(1,\infty)$. Keeping in mind \eqref{PtKt} and using \eqref{sqrtbound} in Proposition \ref{P:boom}, we have for $1\le t<\infty$
\[
|P_t f(X)| \leq \frac{c_N}{V(t)} ||f||_{1}\leq \frac{c_N}{c_0} \frac{||f||_{1}}{\sqrt{t}}.
\]
We thus infer
\begin{align*}
& \frac{s}{\Gamma(1-s)}\left|\int_1^\infty \frac{1}{t^{1+s}} \left((P_t f(X) -e^{-t} f(X)\right) dt\right| \leq \frac{s}{\Gamma(1-s)} \int_1^\infty \frac{1}{t^{1+s}} \left(|P_t f(X)| + e^{-t} |f(X)|\right) dt\\
&\leq \frac{s}{\Gamma(1-s)}\left( \frac{c_N}{c_0}||f||_{1} \int_1^\infty \frac{1}{t^{1+s+\frac{1}{2}}}dt + ||f||_{\infty} \int_1^\infty e^{-t}dt\right)\\
&=\frac{s}{\Gamma(1-s)}\left( \frac{c_N}{c_0}||f||_{1}\frac{2}{2s+1} + ||f||_{\infty}e^{-1}\right)\ \underset{s\to 0^+}{\longrightarrow}\ 0.
\end{align*}
This establishes the desired conclusion in case (i).  To settle the case (ii), suppose that $\max\{\Re(\lambda)\mid \lambda\in \sigma(B)\}< 0$, and denote
$$
m_\infty(f)=\frac{(4\pi)^{-\frac N2}}{\sqrt{{\rm{det}}\ K_\infty}}\int_{\RN} f(Y)e^{-\frac{<K^{-1}_\infty Y,Y>}{4}}dY,
$$
the average of $f$ with respect to the invariant measure. Notice that from well-known Gaussian formulas we have $m_\infty(f)\le ||f||_\infty$.
As before, using \eqref{identity}, we obtain
\begin{align*}
& (-\sA)^s f(X) - f(X) + m_\infty(f)
\\
&= - \frac{s}{\Gamma(1-s)} \int_0^\infty \frac{1}{t^{1+s}} \left((P_t f(X) -  f(X)) + (1-e^{-t}) (f(X) - m_\infty(f))\right) dt\\
& = - \frac{s}{\Gamma(1-s)} \int_0^1 \frac{1}{t^{1+s}} \left((P_t f(X) -  f(X)) + (1-e^{-t}) (f(X) - m_\infty(f))\right) dt
\\
& - \frac{s}{\Gamma(1-s)} \int_1^\infty \frac{1}{t^{1+s}} \left((P_t f(X)-m_\infty(f)) +e^{-t} (m_\infty(f)-f(X))\right) dt.
\end{align*}
The integral on the interval $(0,1)$ can be treated as in the first part of the proof since $|f(X) - m_\infty(f)|\leq 2||f||_\infty$. For the integral on $(1,\infty)$ we exploit the estimate $|P_t f(X)-m_\infty(f)|\leq C_{f,X} e^{-\alpha t}$ established in \eqref{expoconv} of Proposition \ref{equilibrium}. We thus find 
\begin{align*}
& \frac{s}{\Gamma(1-s)}\left|\int_1^\infty \frac{1}{t^{1+s}} \left((P_t f(X)-m_\infty(f)) +e^{-t} (m_\infty(f)-f(X))\right) dt\right| \\
&\leq \frac{s}{\Gamma(1-s)} \int_1^\infty \frac{1}{t^{1+s}} \left(|P_t f(X)-m_\infty(f)| + e^{-t} |m_\infty(f)-f(X)|\right) dt\\
&\leq \frac{s}{\Gamma(1-s)}\left( \int_1^\infty \frac{ C_{f,X} e^{-\alpha t} + 2||f||_\infty e^{-t}}{t^{1+s}}dt\right)\\
&\leq \frac{s}{\Gamma(1-s)}\left( \int_1^\infty \left(C_{f,X} e^{-\alpha t} + 2||f||_\infty e^{-t}\right)dt\right)\underset{s\to 0^+}{\longrightarrow}\ 0.
\end{align*}
This completes the proof.

\end{proof}
}


\section{Some basic properties of the Besov spaces}\label{S:prop}

In this section we prove some basic properties of the Besov spaces $\mathfrak B^\sA_{s,p}$ under the assumption \eqref{trace}. The main result is Proposition \ref{P:density}, which establishes a key density property for such spaces. This generalises the well-known one for the classical spaces $W^{s,p}$, see e.g. \cite[Theor. 7.38]{Ad}.

Since this is not immediately obvious from its definition, we begin by observing that the Besov seminorm introduced in \eqref{sn} does satisfy the following triangle inequality for all $f, g\in \Bps$,
\begin{equation}\label{triangular}
\mathscr N^\sA_{s,p}(f+g)\leq \mathscr N^\sA_{s,p}(f)+\mathscr N^\sA_{s,p}(g).
\end{equation}
To prove \eqref{triangular} we notice that 
$\mathscr N^\sA_{s,p}(f)=||w_f||_{L^p\left(\R^N\times\R^N\times(0,\infty)\right)}$,
where
\begin{equation}\label{defwf}
w_f(X,Y,t)=t^{-\frac s2-\frac1p} p(X,Y,t)^{\frac1p}\left(f(Y)-f(X)\right).
\end{equation}
Inequality \eqref{triangular} thus follows  from the additivity property $w_{f+g}= w_f + w_g$ and the triangle inequality in $L^p\left(\R^N\times\R^N\times(0,\infty)\right)$. A second useful observation concerns what happens to the Besov-type seminorms $\mathscr N^\sA_{s,p}$ when we change the fractional order $s$ of differentiation. 

\begin{lemma}\label{L:es}
Assume \eqref{trace}, and let $p\geq 1$ and $0<s\leq \sigma$. Then, for every $f\in \mathfrak B^\sA_{\sigma,p}$ we have
\begin{equation}\label{bastasigmauno}
 \mathscr N^\sA_{s,p}(f)^p\leq  \mathscr N^\sA_{\sigma,p}(f)^p + \frac{2^{p+1}}{s p} ||f||^p_p.
\end{equation}
In particular, \eqref{bastasigmauno} implies $\mathfrak B^\sA_{\sigma,p}\hookrightarrow \mathfrak B^\sA_{s,p}$. 
\end{lemma}

\begin{proof}
Using \eqref{stochcompl}-\eqref{aggstoch}, together with the hypothesis $0<s\leq \sigma$ and $\tr\geq 0$, we obtain
\begin{align*}
& \mathscr N^\sA_{s,p}(f)^p
\\
&=\int_0^1 \frac{1}{t^{1+\frac{sp}{2}}} \int_{\RN} P_t(|f - f(X)|^p)(X) dX dt + \int_1^\infty \frac{1}{t^{1+\frac{sp}{2}}} \int_{\RN} P_t(|f - f(X)|^p)(X) dX dt\\
&\leq \int_0^1 \frac{1}{t^{1+\frac{\sigma p}{2}}} \int_{\RN} P_t(|f - f(X)|^p)(X) dX dt +\\
&+ \int_1^\infty \frac{2^{p-1}}{t^{1+\frac{sp}{2}}} \int_{\RN}\int_{\RN} p(X,Y,t)\left(|f(Y)|^p + |f(X)|^p\right) dY dX dt\\
&\leq \mathscr N^\sA_{\sigma,p}(f)^p + \int_1^\infty \frac{2^{p-1}}{t^{1+\frac{sp}{2}}}\left( e^{-t\tr}\|f\|_p^p + \|f\|_p^p \right)dt\leq \mathscr N^\sA_{\sigma,p}(f)^p + 2^p\|f\|_p^p \int_1^\infty t^{-1-\frac{sp}{2}}dt\\
&= \mathscr N^\sA_{\sigma,p}(f)^p + \frac{2^{p+1}}{sp}\|f\|^p_p,
\end{align*}
which proves \eqref{bastasigmauno}.

\end{proof}

Next, we recall that from \cite[Lemma 7.3]{GTiso} we know that $\mathscr S\subset \Bps$ for any $0<s<1$ and $1\le p<\infty$. In the next result we prove that, under the assumption \eqref{trace}, the space $C^\infty_0$, and therefore the Schwartz class $\So$, is actually dense in $\Bps$.

\begin{proposition}\label{P:density}
Assume \eqref{trace}. For every $0<s<1$ and $1\le p<\infty$, we have $\overline{C^\infty_0}^{\Bps}=\Bps$.
\end{proposition}

\begin{proof}
\noindent{\it Step I.} We first show that
$\overline{C^{\infty}\cap \Bps}^{\Bps}=\Bps.$
Precisely, we fix $\rho\in C^{\infty}_{0}$, $\text{supp}\ \rho \subseteq\{|Z|\leq 1\}$, $\rho\ge0$ and $||\rho||_1=1$, and consider a family of approximate to the identity $\rho_\ve(Z)=\ve^{-N}\rho\left(\frac{Z}{\ve}\right)$. We shall prove that, remarkably, the standard convolution of a function $f\in \Bps$ with $\rho_{\varepsilon}$ establishes the following result
\begin{equation}\label{denseinfty}
\rho_{\varepsilon}\ast f \to f \mbox{ in }\ \Bps\,\,\, \mbox{ as }\varepsilon\to 0^{+}.
\end{equation}
We mention at this point that, a related local density result for the Sobolev spaces generated by vector fields with Lipschitz coefficients was first discovered by Friedrichs himself in \cite{Fri}, see also \cite[Appendix]{GN} where a global version of this result was found. We thus fix $f\in \Bps$ and denote $f_{\varepsilon}(X)=\left(\rho_{\varepsilon}\ast f\right) (X)=\int_{\mathbb{R}^{N}}f(X-Z)\rho_{\varepsilon}(Z)dZ$. It is classical that $f_{\varepsilon} \in C^{\infty} \cap L^p$ and $||f_{\varepsilon}-f||_p\to 0$ as $\ve\to 0^+$. To prove \eqref{denseinfty} we will show that $\mathscr{N}^{\mathscr{A}}_{s,p}(f-f_{\varepsilon})\to 0$ as $\ve\to 0^+$. This fact, together with \eqref{triangular}, will also tell us that $f_{\varepsilon} \in \Bps$.
We now write
\begin{align*}
\mathscr{N}^{\mathscr{A}}_{s,p}(f-f_{\varepsilon})^p &=\int_{0}^{\infty}t^{-1-\frac{s p}{2} }\int_{\mathbb{R}^{N}}P_t(|f-f_{\varepsilon}-f(X)+f_{\varepsilon}(X)|^{p})(X)dXdt=\\
	&=\int_{0}^{1}t^{-1-\frac{s p}{2}}\int_{\mathbb{R}^{N}}\int_{\mathbb{R}^{N}}p(X,Y,t)|f(Y)-f_{\varepsilon}(Y)-f(X)+f_{\varepsilon}(X)|^{p}dYdXdt+\\
	&+\int_{1}^{\infty}t^{-1-\frac{s p}{2}}\int_{\mathbb{R}^{N}}\int_{\mathbb{R}^{N}}p(X,Y,t)|f(Y)-f_{\varepsilon}(Y)-f(X)+f_{\varepsilon}(X)|^{p}dYdXdt.
\end{align*}
It is easy to see that the last integral tends to $0$ as $\varepsilon\to 0^+$. In fact, by \eqref{stochcompl} and \eqref{aggstoch} we have
\begin{align*}
	&\int_{1}^{\infty}t^{-1-\frac{s p}{2}}\int_{\mathbb{R}^{N}}\int_{\mathbb{R}^{N}}p(X,Y,t)|f(Y)-f_{\varepsilon}(Y)-f(X)+f_{\varepsilon}(X)|^{p}dYdXdt\\
	&\leq 2^{p-1}\left(\int_{1}^{\infty}t^{-1-\frac{s p}{2}}\int_{\mathbb{R}^{N}}\int_{\mathbb{R}^{N}}p(X,Y,t)|f(Y)-f_{\varepsilon}(Y)|^{p}dXdYdt+\right.\\
	&\left.+\int_{1}^{\infty}t^{-1-\frac{s p}{2}}\int_{\mathbb{R}^{N}}\int_{\mathbb{R}^{N}}p(X,Y,t)|f(X)-f_{\varepsilon}(X)|^{p}dYdXdt\right)\\
	&=2^{p-1}\left(\int_{1}^{\infty}t^{-1-\frac{s p}{2}}e^{-t\tr }dt\right)\int_{\mathbb{R}^{N}}|f(Y)-f_{\varepsilon}(Y)|^{p}dY+\\
	&+2^{p-1}\left(\int_{1}^{\infty}t^{-1-\frac{s p}{2}}dt\right)\int_{\mathbb{R}^{N}}|f(X)-f_{\varepsilon}(X)|^{p}dX\ \underset{\ve\to 0^+}{\longrightarrow}\ 0,
\end{align*}
since $\tr \ge 0$, and $||f-f_{\varepsilon}||_{p}\underset{\ve\to 0^+}{\longrightarrow} 0$.  
To complete the proof of \eqref{denseinfty}, we are left with proving that 
\begin{equation}\label{claimsmallt}
\int_{0}^{1}t^{-1-\frac{s p}{2}}\int_{\mathbb{R}^{N}}\int_{\mathbb{R}^{N}}p(X,Y,t)|f(Y)-f_{\varepsilon}(Y)-f(X)+f_{\varepsilon}(X)|^{p}dYdXdt \underset{\ve\to 0^+}{\longrightarrow} 0.
\end{equation}
With this objective in mind, for $X,Y\in \RN$ and $0\leq t\leq 1$ we write
\begin{align*}
&f(Y)-f_{\varepsilon}(Y)-f(X)+f_{\varepsilon}(X)=f(Y)-f(X)-\int_{\RN}\hspace{-0.1cm}f(Y-\ve Z)\rho(Z)dZ+\int_{\RN}\hspace{-0.1cm}f(X-\ve Z)\rho(Z)dZ\\
&=(f(Y)-f(X))-\int_{\RN}\left(f(Y-\ve Z)-f(X-\ve e^{-tB}Z)\right)\rho(Z)dZ
\\
&+\int_{\RN}\left(f(X-\ve Z)-f(X)\right)\rho(Z)dZ- \int_{\RN}\left(f(X-\ve e^{-tB}Z)-f(X)\right)\rho(Z)dZ\\
&=(f(Y)-f(X))-\int_{\RN}\left(f(Y-\ve Z)-f(X-\ve e^{-tB}Z)\right)\rho(Z)dZ
\\
&+\int_{\RN}\left(f(X-\ve Z)-f(X)\right)\left(\rho(Z)-e^{t\tr}\rho(e^{tB}Z)\right)dZ.
\end{align*}
Using \eqref{PtKt} and \eqref{m} we now  observe that the following
identity holds
$$p(X,Y,t)=p(X-\ve e^{-tB}Z,Y-\ve Z,t).$$
Combining these two facts we thus have
\begin{align*}
& p(X,Y,t)|f(Y)-f_{\varepsilon}(Y)-f(X)+f_{\varepsilon}(X)|^{p} = \bigg|p(X,Y,t)^{\frac1p}(f(Y)-f(X))
\\
& - \int_{\RN}\left(p(X-\ve e^{-tB}Z,Y-\ve Z,t)\right)^{\frac1p}\left(f(Y-\ve Z)-f(X-\ve e^{-tB}Z)\right)\rho(Z)dZ
\\
& + \int_{\RN} p(X,Y,t)^{\frac1p}\left(f(X-\ve Z)-f(X)\right)\left(\rho(Z)-e^{t\tr}\rho(e^{tB}Z)\right)dZ\bigg|^{p}.
\end{align*}
Moreover, keeping \eqref{defwf} in mind, and using $\text{supp}\ \rho\subseteq \{|Z|\leq 1\}$, $||\rho||_1=1$, and H\"older's inequality, we find
\begin{align*}
&t^{-1-\frac{s p}{2}} p(X,Y,t) |f(Y)-f_{\varepsilon}(Y)-f(X)+f_{\varepsilon}(X)|^{p}\\
&= \bigg|w_f(X,Y,t)-\int_{\RN}w_f(X-\ve e^{-tB}Z,Y-\ve Z,t)\rho(Z)dZ
\\
& + t^{-\frac1p-\frac s 2}\int_{\RN} p(X,Y,t)^{\frac1p}\left(f(X-\ve Z)-f(X)\right)\left(\rho(Z)-e^{t\tr}\rho(e^{tB}Z)\right)dZ \bigg|^{p}
\\
&= \bigg|\int_{\{|Z|\leq 1\}}\left(w_f(X,Y,t)-w_f(X-\ve e^{-tB}Z,Y-\ve Z,t)\right)\rho(Z)dZ
\\
& + t^{-\frac1p-\frac s 2}\int_{\RN} p(X,Y,t)^{\frac1p}\left(f(X-\ve Z)-f(X)\right)\left(\rho(Z)-e^{t\tr}\rho(e^{tB}Z)\right)dZ \bigg|^{p}\\
&\leq 2^{p-1}\bigg(|\{|Z|\leq 1\}|^{p-1}\int_{\{|Z|\leq 1\}}\left|w_f(X,Y,t)-w_f(X-\ve e^{-tB}Z,Y-\ve Z,t)\right|^p\rho(Z)^p dZ 
\\
&+ |\{|Z|\leq M\}|^{p-1}\int_{\{|Z|\leq M\}}t^{-1-\frac{s p}{2}}p(X,Y,t)\left|f(X-\ve Z)-f(X)\right|^p\left|\rho(Z)-e^{t\tr}\rho(e^{tB}Z)\right|^p dZ\bigg),
\end{align*}
where $M\geq 1$ is such that $|e^{-tB}Z|\leq M$ for all $|Z|\leq 1$ and $0\leq t\leq 1$. Hence, by \eqref{stochcompl} and the previous inequality, we have
\begin{align}\label{newterm}
&\int_{0}^{1}t^{-1-\frac{s p}{2}}\int_{\mathbb{R}^{N}}\int_{\mathbb{R}^{N}}p(X,Y,t)|f(Y)-f_{\varepsilon}(Y)-f(X)+f_{\varepsilon}(X)|^{p}dYdXdt
\\
& \leq 2^{p-1}|B_1|^{p-1}\bigg(\int_{B_{1}}\rho^p(Z)\int_0^1||w_f(\cdot,\cdot,t)-w_f(\cdot-\ve e^{-tB}Z,\cdot-\ve Z,t)||^p_{L^p\left(\RN\times\RN\right)}dtdZ 
\notag
\\
& + |\{|Z|\leq M\}|^{p-1}\int_{\{|Z|\leq M\}}||f(\cdot-\ve Z)-f||_p^p\int_0^1\frac{\left|\rho(Z)-e^{t\tr}\rho(e^{tB}Z)\right|^p}{t^{1+\frac{s p}{2}}}dtdZ\bigg).
\notag
\end{align}
We explicitly remark that the last term containing $\left(\rho(Z)-e^{t\tr}\rho(e^{tB}Z)\right)$ does not appear when $B=0$. To complete the proof of \eqref{claimsmallt} we next show that the two integrals in the right-hand side of \eqref{newterm} converge to $0$ as $\ve\to 0^+$. To see that the second integral goes to zero we observe that $t^{-1-\frac{s p}{2}}\left|\rho(Z)-e^{t\tr}\rho(e^{tB}Z)\right|^p$ is summable on $[0,1]$ since $\rho(Z)-e^{t\tr}\rho(e^{tB}Z)=O(t)$ as $t\to 0$, uniformly in $|Z|\le M$. On the other hand, $f\in L^p$  implies that $||f(\cdot-\ve Z)-f||_p^p \le 2^p||f||^p_p$. 
By Lebesgue dominated convergence we conclude that $||f(\cdot-\ve Z)-f||_p^p \underset{\ve\to 0^+}{\longrightarrow} 0$. 
To recognise that the first integral in \eqref{newterm} converges to zero we observe that $f\in \Bps$ is equivalent to saying that $w_f \in L^p\left(\R^N\times\R^N\times(0,\infty)\right)$, see \eqref{defwf}. Therefore, by the boundedness of $e^{-tB}Z$ for $|Z|\le 1$ and $t\in [0,1]$ and the continuity in $L^p$ mean, for almost any $t\in (0,1)$ we have
$$||w_f(\cdot,\cdot,t)-w_f(\cdot-\ve e^{-tB}Z,\cdot-\ve Z,t)||^p_{L^p\left(\RN\times\RN\right)} \underset{\ve\to 0^+}{\longrightarrow} 0.$$
Keeping in mind that $||w_f(\cdot,\cdot,t)-w_f(\cdot-\ve e^{-tB}Z,\cdot-\ve Z,t)||^p_{L^p\left(\RN\times\RN\right)}\le 2^p||w_f(\cdot,\cdot,t)||^p_{L^p\left(\RN\times\RN\right)} \in L^1(0,1)$, by Lebesgue dominated convergence we conclude that also the first integral in \eqref{newterm} converges to zero as $\ve\to 0^+$.
This completes the proof of \eqref{denseinfty}.

\vskip 0.4cm

\noindent{\it Step II. } We finish the proof of the proposition by showing that
$
\overline{ C^{\infty}_{0}}^{\Bps}=\Bps.
$
With Step I in hands, it is now enough to show that if $f\in \Bps$, and $\{\eta_{\varepsilon}\}_{\varepsilon>0}$ is a family of smooth cut-off functions approximating $1$ in a pointwise sense, then we have $\eta_\ve f \underset{\ve \to 0^+}{\longrightarrow}  f$ in $\Bps$.
More precisely, let $\eta_\ve(Z)=\eta(\varepsilon Z)$, where $\eta\in C^{\infty}_{0}$ is such that $0\leq\eta\leq 1$, $\eta(Z) \equiv 1$ for $|Z|\le 1$ and $\eta\equiv 0$ for $|Z|\ge 2$. It is trivial that $||\eta_\ve f-f||_p \to 0$ as $\ve\to 0^+$. Moreover, we have
\begin{align*}
& \mathscr{N}^{\mathscr{A}}_{s,p}(f-\eta_{\varepsilon}f)^{p} = \int_{0}^{\infty}\int_{\mathbb{R}^{N}}\int_{\mathbb{R}^{N}}\frac{p(X,Y,t)}{t^{1+\frac{s p}{2}}}\left|f(Y)-\eta(\ve Y)f(Y)-f(X)+\eta(\ve X)f(X)\right|^{p}dYdXdt
\\
&\overset{def}{=}\int_{0}^{\infty}\int_{\mathbb{R}^{N}}\int_{\mathbb{R}^{N}}g_\ve(X,Y,t)dYdXdt.
	\end{align*}
It is easy to recognise that $g_\ve(X,Y,t)\underset{\ve\to 0^+}{\longrightarrow} 0$, for almost every $(X,Y,t)\in \RN\times\RN\times(0,\infty)$. We also notice that in view of \eqref{stochcompl}, \eqref{aggstoch}, the fact that $f\in L^p$, and that $s p>0$, we have for large values of $t$ 
\begin{align*}
&\mathbf 1_{(1,\infty)}(t)g_\ve(X,Y,t)\leq 
2^{p-1}\mathbf 1_{(1,\infty)}(t)t^{-1-\frac{s p}{2}}p(X,Y,t)\left((1-\eta(\ve Y))^p|f(Y)|^p+(1-\eta(\ve X))^p|f(X)|^p\right)\\
&\leq 2^{p-1}\mathbf 1_{(1,\infty)}(t)t^{-1-\frac{s p}{2}}p(X,Y,t)\left(|f(Y)|^p+|f(X)|^p\right)\in L^1\left(\RN\times\RN\times(0,\infty)\right).
\end{align*}
On the other hand, if we indicate $B_2 = \{Z\in \RN\mid |Z|\le 2\}$, then for small values of $t$ and every $0<\ve\leq 1$, we have
\begin{align*}
&\mathbf 1_{(0,1)}(t)g_\ve(X,Y,t) \\
&\leq\mathbf 1_{(0,1)}(t)c_p t^{-1-\frac{s p}{2}}p(X,Y,t)\left((1+\eta^p(\ve Y))|f(Y)-f(X)|^p+\left|\eta(\ve X)-\eta(\ve Y)\right|^p|f(X)|^p\right)\\
&\leq \mathbf 1_{(0,1)}(t)c_p\left(2|w_f(X,Y,t)|^p+ \left(\mathbf 1_{B_2}(\ve X)+\mathbf 1_{B_2}(\ve Y)\right)\ve^p||\nabla\eta||^p_\infty \frac{p(X,Y,t)}{t^{1+\frac{s p}{2}}}\left|Y-X\right|^p|f(X)|^p\right)\\
&\leq  2c_p |w_f(X,Y,t)|^p + c_p||\nabla\eta||^p_\infty\mathbf 1_{(0,1)}(t)t^{-1-\frac{s p}{2}}|f(X)|^p p(X,Y,t)\left(\mathbf 1_{B_2}(\ve X)\ve^p2^{p-1}\left(\left|Y-e^{tB}X\right|^p+\right.\right.\\
&\left.\left.+\left|\left(e^{tB}-\mathbb{I}\right)X\right|^p\right) +\mathbf 1_{B_2}(\ve Y)\ve^p2^{p-1}\left(\left|e^{-tB}Y-X\right|^p+\left|\left(\mathbb{I}-e^{-tB}\right)Y\right|^p\right)\right)\\
&\leq  2c_p |w_f(X,Y,t)|^p + c_p2^{p-1}||\nabla\eta||^p_\infty\mathbf 1_{(0,1)}(t)t^{-1-\frac{s p}{2}}|f(X)|^p p(X,Y,t)\left|Y-e^{tB}X\right|^p\left(1+\|e^{-tB}\|^p\right) +\\
&+c_p2^{p-1}||\nabla\eta||^p_\infty\mathbf 1_{(0,1)}(t)t^{-1-\frac{s p}{2}}|f(X)|^p p(X,Y,t)\left(\|e^{tB}-\mathbb{I}\|^p\left|\ve X\right|^p\mathbf 1_{B_2}(\ve X) +\right.\\
&\left.+ \|\mathbb{I}-e^{-tB}\|^p\left|\ve Y\right|^p\mathbf 1_{B_2}(\ve Y)\right)\\
&\leq  2c_p |w_f(X,Y,t)|^p + c_p2^{p-1}||\nabla\eta||^p_\infty\mathbf 1_{(0,1)}(t)t^{-1-\frac{s p}{2}}|f(X)|^p p(X,Y,t)\left|Y-e^{tB}X\right|^p\left(1+\|e^{-tB}\|^p\right) +\\
&+c_p2^{2p-1}||\nabla\eta||^p_\infty\mathbf 1_{(0,1)}(t)t^{-1-\frac{s p}{2}}|f(X)|^p p(X,Y,t)\left(\|e^{tB}-\mathbb{I}\|^p+ \|\mathbb{I}-e^{-tB}\|^p\right).
\end{align*}
The previous chain of inequalities shows an uniform bound in $\ve$ for $\mathbf 1_{(0,1)}(t)g_\ve(X,Y,t)$ in terms of a sum of three functions. The first function belongs to $L^1\left(\RN\times\RN\times(0,\infty)\right)$ since $f\in \mathfrak B^\sA_{s,p}$ and thus $w_f\in L^p$. The last one belongs to $L^1$ by \eqref{stochcompl}, the fact that $f\in L^p$, and
$\left(\|e^{tB}-\mathbb{I}\|^p+ \|\mathbb{I}-e^{-tB}\|^p\right)=O(t^p)\quad\mbox{ as }t\to 0^+.$
Finally, also the second function belongs to $L^1$ since, in view of the fact that $\left(1+\|e^{-tB}\|^p\right)$ stays bounded for $0\leq t\leq 1$, that $0<s<1$ and that $\|\sqrt{K(t)}\|$ is uniformly bounded for $0\leq t\leq 1$, we have
\begin{align*}
&\int_0^1\int_{\RN}\int_{\RN}t^{-1-\frac{s p}{2}}|f(X)|^p p(X,Y,t)\left|Y-e^{tB}X\right|^p dY dX dt\\
&=\int_0^1\int_{\RN}\int_{\RN} t^{-1-\frac{s p}{2}}|f(X)|^p p(0,\xi,1)\left|\sqrt{tK(t)}\xi\right|^p d\xi dX dt\\
&=||f||^p_p \int_0^1\int_{\RN} \frac{t^{\frac p2 (1-s)}}{t}p(0,\xi,1)\left|\sqrt{K(t)}\xi\right|^pd\xi dt<\infty.
\end{align*}
All these considerations, and Lebesgue dominated convergence theorem, allow to conclude that $\mathscr{N}^{\mathscr{A}}_{s,p}(f-\eta_{\varepsilon}f)^{p} \underset{\ve\to 0^+}{\longrightarrow} 0$. This completes the proof of Step II.

\end{proof}



\section{Limiting behaviour of the Besov spaces as $s\to 0^+$:  Proof of Theorem \ref{T:MSallp}}\label{S:lim}

With the preliminary work of the previous sections in place, in the present one we can finally establish our generalisation of the result by Maz'ya \& Shaposhnikova \eqref{MSheat} to the Besov spaces $\mathfrak B_{s,p}^\sA$. The following four lemmas constitute the core of the proof of Theorem \ref{T:MSallp}.

\begin{lemma}\label{L:uno}
Let $1\leq p <\infty$, and suppose $f\in \underset{0<\sigma<1}{\bigcup}\mathfrak B^\sA_{\sigma,p}$. Then,
\begin{equation}\label{stateuno}
\underset{s\to 0^+}{\lim} s \int_0^1  \frac{1}{t^{\frac{s p}2 +1}} \int_{\RN} P_t\left(|f - f(X)|^p\right)(X) dX dt = 0.
\end{equation}
\end{lemma}
\begin{proof}
Suppose $f\in \mathfrak B^\sA_{\sigma,p}$ for some $\sigma \in (0,1)$. For $0<s\leq \sigma$, we have
\begin{align*}
&\int_0^1  \frac{1}{t^{\frac{s p}2 +1}} \int_{\RN} P_t\left(|f - f(X)|^p\right)(X) dX dt\\
&\leq \int_0^1  \frac{1}{t^{\frac{\sigma p}2 +1}} \int_{\RN} P_t\left(|f - f(X)|^p\right)(X) dX dt \leq \mathscr N^\sA_{\sigma,p}(f)^p<\infty.
\end{align*}
Being the previous inequality valid for all $s\leq \sigma$, \eqref{stateuno} easily follows by multiplying by $s$ and passing to the limit.

\end{proof}

\begin{lemma}\label{L:due}
Let $1\leq p <\infty$ and suppose $f\in L^p$. Then,
\begin{equation}\label{statedue}
\underset{s\to 0^+}{\lim}  s \int_1^\infty\int_{\RN}\int_{\RN}  \frac{p(X,Y,t)}{t^{\frac{s p}2 +1}}\left(|f(X)|^p+|f(Y)|^p\right)dY dX dt = \begin{cases}
\frac{4}{p} ||f||_p^p,\ \ \  \mbox{if }  \tr = 0,
\\
\\
\frac{2}{p}||f||_p^p,\ \ \ \mbox{if } \tr >0.
\end{cases}
\end{equation}
\end{lemma}
\begin{proof}
By \eqref{stochcompl} and \eqref{aggstoch} we have
\begin{align*}
&s\int_1^\infty \int_{\RN}\int_{\RN} \frac{1}{t^{\frac{s p}2 +1}}p(X,Y,t)\left(|f(X)|^p+|f(Y)|^p\right)dYdX dt\\
&=s\int_1^\infty \int_{\RN}\frac{1}{t^{\frac{s p}2 +1}}|f(X)|^p P_t 1(X) dX dt + s\int_1^\infty \int_{\RN}\frac{1}{t^{\frac{s p}2 +1}}|f(Y)|^p P^*_t 1(Y) dYdt\\
&=s||f||^p_p \int_1^\infty\frac{1+e^{-t\tr}}{t^{\frac{s p}2 +1}}dt.
\end{align*}
If $\tr =0$ the desired conclusion readily follows from the previous identity. If instead $\tr>0$, it is enough to notice that 
$0\leq \int_1^\infty\frac{e^{-t\tr}}{t^{\frac{s p}2 +1}}dt\leq \int_1^\infty e^{-t\tr}dt=\frac{e^{-\tr}}{\tr},$
which implies
\begin{equation}\label{strB}
s\int_1^\infty\frac{e^{-t\tr}}{t^{\frac{s p}2 +1}}dt\underset{s\to 0^+}{\longrightarrow} 0,
\end{equation}
and concludes the proof.

\end{proof}

\begin{lemma}\label{L:tre}
Let $1\leq p <\infty$ and suppose $\tr >0$. If $f\in L^p$, then
\begin{equation}\label{statetre}
\underset{s\to 0^+}{\lim} s \int_1^\infty  \frac{1}{t^{\frac{s p}2 +1}} \int_{\RN} P_t\left(|f - f(X)|^p\right)(X) dX dt = \frac{2}{p}||f||_p^p.
\end{equation}
\end{lemma}
\begin{proof}
Assume first that $p>1$. 
We begin by observing that, for $f\in L^p$, we have
\begin{equation}\label{restip}
\underset{s\to 0^+}{\lim} s \int_1^\infty\int_{\RN}\int_{\RN}  \frac{p(X,Y,t)}{t^{\frac{s p}2 +1}}\left(|f(X)|^{p-1}|f(Y)|+|f(X)||f(Y)|^{p-1}\right)dY dX dt=0.
\end{equation}
To see \eqref{restip}, we observe that H\"older inequality and Proposition \ref{P:ptq} (applied with $q=p$, and $q=p'$), imply
\begin{align*}
&0\leq s \int_1^\infty\int_{\RN}\int_{\RN}  \frac{p(X,Y,t)}{t^{\frac{s p}2 +1}}\left(|f(X)|^{p-1}|f(Y)|+|f(X)||f(Y)|^{p-1}\right)dY dX dt\\
&=s\int_1^\infty\frac{1}{t^{\frac{s p}2 +1}}\int_{\RN}\left(|f(X)|^{p-1}P_t(|f|)(X)+|f(X)|P_t(|f|^{p-1})(X)\right)dX dt\\
&\leq s\int_1^\infty\frac{1}{t^{\frac{s p}2 +1}}\left(||f||_p^{p-1}||P_t(|f|)||_p+||f||_p||P_t(|f|^{p-1})||_{p'}\right)dt\\
&\leq s\int_1^\infty\frac{1}{t^{\frac{s p}2 +1}}\left(C(p)e^{-t\frac{\tr}{p}}||f||_p^{p}+C(p')e^{-t\frac{\tr}{p'}}||f||^p_p\right)dt\le \overline C(p)\  ||f||_p^{p}\ s \int_1^\infty\frac{e^{-t\frac{\tr}{p}}+e^{-t\frac{\tr}{p'}}}{t^{\frac{s p}2 +1}}dt.
\end{align*}
Arguing exactly as in \eqref{strB}, we see that the last term tends to $0$ as $s\to 0^+$. This shows \eqref{restip}.
To prove \eqref{statetre} we next exploit the following simple fact: there exists a positive constant $C_p$ such that
\begin{equation}\label{calculus}
||a-b|^p-|a|^p-|b|^p|\leq C_p\left(|a|^{p-1}|b|+|a||b|^{p-1}\right)\,\quad\mbox{ for all }a,b\in\R.
\end{equation}
This can be checked by noticing that the function $h:\R\smallsetminus\{0\}\longrightarrow\R$, 
$h(x)=\frac{|x-1|^p-|x|^p-1}{|x|^{p-1}+|x|}$,
has finite limits at $x=0^{\pm}$ and $x=\pm\infty$ and thus, in particular, it is globally bounded. Applying \eqref{calculus} with the choices $a=f(X)$ and $b=f(Y)$, we find
\begin{align*}
& \left|s \int_1^\infty \int_{\RN} \frac{P_t\left(|f - f(X)|^p\right)(X)}{t^{\frac{s p}2 +1}} dX dt  - s \int_1^\infty\int_{\RN}\int_{\RN}  \frac{p(X,Y,t)}{t^{\frac{s p}2 +1}} \left(|f(X)|^p + |f(Y)|^p\right) dY dX dt\right|\\
&= s \left|\int_1^\infty\int_{\RN}\int_{\RN}  \frac{p(X,Y,t)}{t^{\frac{s p}2 +1}}\left(|f(X)-f(Y)|^p-|f(X)|^p - |f(Y)|^p\right)dY dX dt\right|
\\
& \leq C_p\ s \int_1^\infty\int_{\RN}\int_{\RN}  \frac{p(X,Y,t)}{t^{\frac{s p}2 +1}} (|f(X)|^{p-1} |f(Y)| + |f(X)| |f(Y)|^{p-1}) dY dX dt.
\end{align*}
From this estimate and from \eqref{restip} we deduce that
\begin{align*}
&\underset{s\to 0^+}{\lim}  s \int_1^\infty  \frac{1}{t^{\frac{s p}2 +1}} \int_{\RN} P_t\left(|f - f(X)|^p\right)(X) dX dt\\
&=\underset{s\to 0^+}{\lim}  s \int_1^\infty\int_{\RN}\int_{\RN}  \frac{p(X,Y,t)}{t^{\frac{s p}2 +1}}\left(|f(X)|^p+|f(Y)|^p\right)dY dX dt.
\end{align*}
At this point, the desired conclusion \eqref{statetre} follows from \eqref{statedue} in Lemma \ref{L:due} in the case $p>1$. We thus turn the attention to the case $p=1$. For any $s>0$ we have by \eqref{stochcompl} 
$$
2||f||_1=s\int_1^{\infty}\frac{1}{t^{\frac{s}2 +1}}dt \int_{\RN}|f(X)|dX=s\int_1^{\infty}\int_{\RN}\int_{\RN}\frac{p(X,Y,t)}{t^{\frac{s}2 +1}}|f(X)|dYdXdt.
$$
This gives
\begin{align*}
&\left|s \int_1^\infty  \frac{1}{t^{\frac{s}2 +1}} \int_{\RN} P_t\left(|f - f(X)|\right)(X) dX dt - 2||f||_1\right|\\
&=s\left|\int_1^\infty\int_{\RN}\int_{\RN}  \frac{p(X,Y,t)}{t^{\frac{s}2 +1}} \left(|f(Y) - f(X)|-|f(X)|\right) dYdX dt\right|\\
&\leq s\int_1^\infty\int_{\RN}\int_{\RN}  \frac{p(X,Y,t)}{t^{\frac{s}2 +1}} |f(Y)|dYdX dt= ||f||_1\ s \int_1^\infty  \frac{e^{-t\tr}}{t^{\frac{s}2 +1}}dt,
\end{align*}
where in the last inequality we have used
\eqref{aggstoch}. From this estimate and \eqref{strB}, we see that \eqref{statetre} holds true also when $p=1$.

\end{proof}

\begin{lemma}\label{L:quattro}
Let $1\leq p <\infty$ and suppose $\tr =0$. If $f\in \So$, then
\begin{equation}\label{statequattro}
\underset{s\to 0^+}{\lim} s \int_1^\infty  \frac{1}{t^{\frac{s p}2 +1}} \int_{\RN} P_t\left(|f - f(X)|^p\right)(X) dX dt = \frac{4}{p}||f||_p^p.
\end{equation}
\end{lemma}
\begin{proof}
We begin by assuming $p>1$. Following the strategy of the proof of Lemma \ref{L:tre}, our aim is to prove 
\begin{equation}\label{restip0}
\underset{s\to 0^+}{\lim} s \int_1^\infty\int_{\RN}\int_{\RN}  \frac{p(X,Y,t)}{t^{\frac{s p}2 +1}}\left(|f(X)|^{p-1}|f(Y)|+|f(X)||f(Y)|^{p-1}\right)dY dX dt=0,
\end{equation}
see \eqref{restip}.
The main difference at this point consists in the fact that, being $\tr =0$, the ultracontractive estimate \eqref{tuttenorme} in Proposition \ref{P:ptq} no longer implies a decay of the semigroup in $L^p$ or $L^{p'}$ . On the other hand, since  $f\in\So$, it is in every $L^q$, and therefore we can combine the  $L^1\to L^p$ and $L^{1}\to L^{p'}$ decays in Proposition \ref{P:ptq} with the critical information contained in \eqref{sqrtbound} of Proposition \ref{P:boom}, and infer
{\allowdisplaybreaks
\begin{align*}
& \ \ \ s \int_1^\infty\int_{\RN}\int_{\RN}  \frac{p(X,Y,t)}{t^{\frac{s p}2 +1}}\left(|f(X)|^{p-1}|f(Y)|+|f(X)||f(Y)|^{p-1}\right)dY dX dt\\
&=s\int_1^\infty\frac{1}{t^{\frac{s p}2 +1}}\int_{\RN}\left(|f(X)|^{p-1}P_t(|f|)(X)+|f(X)|P_t(|f|^{p-1})(X)\right)dX dt\\
&\leq s\int_1^\infty\frac{1}{t^{\frac{s p}2 +1}}\left(||f||_p^{p-1}||P_t(|f|)||_p+||f||_p||P_t(|f|^{p-1})||_{p'}\right)dt
\\
& \leq C(p) s \int_1^\infty\frac{1}{t^{\frac{s p}2 +1}}\left(\frac{1}{V(t)^{1-\frac 1p}}||f||_1||f||_p^{p-1}+\frac{1}{V(t)^{1-\frac{1}{p'}}}|||f|^{p-1}||_1||f||_p\right)dt\\
& \le C'(p) s \left(||f||_1||f||_p^{p-1}\int_1^\infty\frac{1}{t^{\frac{s p}2 +1 + \frac{1}{2p'}}}dt+ |||f|^{p-1}||_1||f||_p\int_1^\infty\frac{1}{t^{\frac{s p}2 +1 + \frac{1}{2p}}}dt\right)\\
&\le 2 C'(p) s\left(\frac{1}{sp+\frac{1}{p'}} ||f||_1||f||_p^{p-1} + \frac{1}{sp+\frac{1}{p}} |||f|^{p-1}||_1||f||_p\right).
\end{align*}
}
Since the last term tends to $0$ as $s\to 0^+$, we conclude that \eqref{restip0} does hold. At this point we argue as in the proof of Lemma \ref{L:tre}. Using \eqref{calculus}, we deduce from \eqref{restip0} that
{\allowdisplaybreaks
\begin{align*}
&\underset{s\to 0^+}{\lim}  s \int_1^\infty\int_{\RN}\int_{\RN}  \frac{p(X,Y,t)}{t^{\frac{s p}2 +1}}|f(X)-f(Y)|^pdY dX dt\\
&=\underset{s\to 0^+}{\lim}  s \int_1^\infty\int_{\RN}\int_{\RN}  \frac{p(X,Y,t)}{t^{\frac{s p}2 +1}}\left(|f(X)|^p+|f(Y)|^p\right)dY dX dt.
\end{align*}}
We know from \eqref{statedue} in Lemma \ref{L:due} that the common value of the previous limits is $\frac{4}{p}||f||_p^p$. This proves the desired conclusion \eqref{statequattro} in the case $p>1$.

We are left with analysing the case $p=1$. By \eqref{stochcompl} and \eqref{aggstoch} (recall that we are assuming $\tr=0$), we have
{\allowdisplaybreaks
\begin{align*}
&  s \int_1^\infty  \frac{1}{t^{\frac{s}2 +1}} \int_{\RN} P_t\left(|f - f(X)|\right)(X) dX dt\leq s\int_1^\infty\int_{\RN}\int_{\RN}  \frac{p(X,Y,t)}{t^{\frac{s}2 +1}} \left(|f(Y)|+|f(X)|\right) dYdX dt\\
&=2s||f||_1\int_1^\infty  \frac{1}{t^{\frac{s}2 +1}}dt=4||f||_1.
\end{align*}}
This trivially implies
$$\underset{s\to 0^+}{\limsup} \,\,\, s\int_1^\infty  \frac{1}{t^{\frac{s}2 +1}} \int_{\RN} P_t\left(|f - f(X)|\right)(X) dX dt\leq 4||f||_1.$$
In order to finish the proof of the lemma, we are left with showing that
\begin{equation}\label{claiminf}
\underset{s\to 0^+}{\liminf} \,\,\, s\int_1^\infty  \frac{1}{t^{\frac{s}2 +1}} \int_{\RN} P_t\left(|f - f(X)|\right)(X) dX dt\geq 4||f||_1.
\end{equation}
With this objective in mind, fix $\ve>0$. Since $f\in L^1$, we can find a compact set $K_\ve\subset \RN$ such that
\begin{equation}\label{decideK}
\int_{\RN\smallsetminus K_\ve}|f(\xi)|d\xi \leq \ve.
\end{equation}
We now have
{\allowdisplaybreaks
\begin{align*}
& s \int_1^\infty  \frac{1}{t^{\frac{s}2 +1}} \int_{\RN} P_t\left(|f - f(X)|\right)(X) dX dt=s\int_1^\infty\frac{1}{t^{\frac{s}2 +1}}\int_{\RN}\int_{\RN}p(X,Y,t)|f(Y) - f(X)|dY dX dt\\
&\geq s\int_1^\infty\frac{1}{t^{\frac{s}2 +1}}\int_{K_\ve}\int_{\RN\smallsetminus K_\ve}p(X,Y,t)|f(Y) - f(X)|dY dX dt +\\
&+s\int_1^\infty\frac{1}{t^{\frac{s}2 +1}}\int_{\RN\smallsetminus K_\ve} \int_{K_\ve} p(X,Y,t)|f(Y) - f(X)|dY dX dt\\
&\geq s\int_1^\infty\frac{1}{t^{\frac{s}2 +1}}\int_{K_\ve}\int_{\RN\smallsetminus K_\ve}p(X,Y,t)\left(|f(X)| - |f(Y)|\right)dY dX dt \\
&+s\int_1^\infty\frac{1}{t^{\frac{s}2 +1}}\int_{\RN\smallsetminus K_\ve} \int_{K_\ve} p(X,Y,t)\left(|f(Y)| - |f(X)|\right)dY dX dt\\
&= s\int_1^\infty\frac{1}{t^{\frac{s}2 +1}}\int_{K_\ve}|f(X)|\left(1-\int_{K_\ve}p(X,Y,t)dY\right) dX dt \\
&+s\int_1^\infty\frac{1}{t^{\frac{s}2 +1}} \int_{K_\ve}|f(Y)|\left(1-\int_{K_\ve} p(X,Y,t) dX\right)dY dt\\
&-s\int_1^\infty\frac{1}{t^{\frac{s}2 +1}}\int_{K_\ve}\int_{\RN\smallsetminus K_\ve}p(X,Y,t)|f(Y)|dY dX dt \\
&-s\int_1^\infty\frac{1}{t^{\frac{s}2 +1}}\int_{\RN\smallsetminus K_\ve} \int_{K_\ve} p(X,Y,t)|f(X)|dY dX dt,
\end{align*}
where in the last equality we used \eqref{stochcompl} and \eqref{aggstoch}. We can rewrite the previous inequality as follows
\begin{align}\label{puguuno}
& s\int_1^\infty  \frac{1}{t^{\frac{s}2 +1}} \int_{\RN} P_t\left(|f - f(X)|\right)(X) dX dt
\\
&\geq s\int_1^\infty\frac{1}{t^{\frac{s}2 +1}}dt\int_{K_\ve}|f(X)|dX+s\int_1^\infty\frac{1}{t^{\frac{s}2 +1}}dt \int_{K_\ve}|f(Y)|dY
\nonumber\\
&-s\int_1^\infty\frac{1}{t^{\frac{s}2 +1}}\int_{K_\ve}|f(X)|\int_{K_\ve}p(X,Y,t)dYdX dt \notag\\
&-s\int_1^\infty\frac{1}{t^{\frac{s}2 +1}} \int_{K_\ve}|f(Y)|\int_{K_\ve} p(X,Y,t) dXdY dt
\nonumber\\
&-s\int_1^\infty\frac{1}{t^{\frac{s}2 +1}}\int_{\RN\smallsetminus K_\ve}|f(Y)|\int_{K_\ve}p(X,Y,t)dX dY dt 
\nonumber\\
&-s\int_1^\infty\frac{1}{t^{\frac{s}2 +1}}\int_{\RN\smallsetminus K_\ve}|f(X)| \int_{K_\ve} p(X,Y,t)dY dX dt.
\nonumber
\end{align}
By \eqref{decideK}, together with \eqref{stochcompl}, \eqref{aggstoch}, we know that}
{\allowdisplaybreaks
\begin{align*}
&s\int_1^\infty\frac{1}{t^{\frac{s}2 +1}}\int_{\RN\smallsetminus K_\ve}|f(Y)|\int_{K_\ve}p(X,Y,t)dX dY dt \\
&+s\int_1^\infty\frac{1}{t^{\frac{s}2 +1}}\int_{\RN\smallsetminus K_\ve}|f(X)| \int_{K_\ve} p(X,Y,t)dY dX dt\\
&\leq s\int_1^\infty\frac{1}{t^{\frac{s}2 +1}} dt \int_{\RN\smallsetminus K_\ve}|f(Y)|dY +\int_1^\infty\frac{1}{t^{\frac{s}2 +1}} dt\int_{\RN\smallsetminus K_\ve}|f(X)| dX\leq 4\ve.
\end{align*}}
On the other hand, using the expression \eqref{PtKt} of $p(X,Y,t)$ we obtain
{\allowdisplaybreaks
\begin{align*}
&s\int_1^\infty\frac{1}{t^{\frac{s}2 +1}}\int_{K_\ve}|f(X)|\int_{K_\ve}p(X,Y,t)dYdX dt+s\int_1^\infty\frac{1}{t^{\frac{s}2 +1}} \int_{K_\ve}|f(Y)|\int_{K_\ve} p(X,Y,t) dXdY dt\\
&\leq c_N\ s |K_\ve| \int_1^\infty\frac{dt}{t^{\frac{s}2 +1}V(t)} \int_{K_\ve}|f(X)|dX+c_N\ s |K_\ve|\int_1^\infty\frac{dt}{t^{\frac{s}2 +1}V(t)} \int_{K_\ve}|f(Y)|dY\\
&\leq 2s \frac{c_N}{c_0}|K_\ve|\ ||f||_1\int_1^\infty\frac{dt}{t^{s +1}}  =\frac{4s}{s+1} \frac{c_N}{c_0}|K_\ve|\ ||f||_1,
\end{align*}
where in the last inequality we have used \eqref{sqrtbound} in Proposition \ref{P:boom}.
Inserting the previous two estimates in \eqref{puguuno}, and using again \eqref{decideK} we deduce
{\allowdisplaybreaks
\begin{align*}
&s\int_1^\infty  \frac{1}{t^{\frac{s}2 +1}} \int_{\RN} P_t\left(|f - f(X)|\right)(X) dX dt\\
&\geq 2s\int_1^\infty\frac{1}{t^{\frac{s}2 +1}}dt\int_{K_\ve}|f(X)|dX-\frac{4s}{s+1} \frac{c_N}{c_0}|K_\ve|\ ||f||_1-4\ve\\
&\geq 4\left(||f||_1 - \ve\right)-\frac{4s}{s+1} \frac{c_N}{c_0}|K_\ve|||f||_1-4\ve= 4||f||_1 - 8\ve-\frac{4s}{s+1} \frac{c_N}{c_0}|K_\ve|\ ||f||_1,
\end{align*}}
which implies
$$
\underset{s\to 0^+}{\liminf} \,\,\, s\int_1^\infty  \frac{1}{t^{\frac{s}2 +1}} \int_{\RN} P_t\left(|f - f(X)|\right)(X) dX dt\geq 4||f||_1 - 8\ve.
$$
The arbitrariness of $\ve$ concludes the proof of \eqref{claiminf}, and of the lemma as well.}

\end{proof}

We are finally in a position to provide the

\begin{proof}[Proof of Theorem \ref{T:MSallp}]
Let $p\ge 1$ and assume that $f\in \underset{0<\sigma<1}{\bigcup}\mathfrak B^\sA_{\sigma,p}$. Suppose that $\sigma \in (0,1)$ is such that $f\in \mathfrak B^\sA_{\sigma,p}$. As before, for every $0<s\le \sigma$ we write
{\allowdisplaybreaks
\begin{align*}
s \mathscr N^\sA_{s,p}(f)^p & =
s \int_0^1  \frac{1}{t^{\frac{s p}2 +1}} \int_{\RN} P_t\left(|f - f(X)|^p\right)(X) dX dt
\\
& + s \int_1^{\infty}  \frac{1}{t^{\frac{s p}2 +1}} \int_{\RN} P_t\left(|f - f(X)|^p\right)(X) dX dt.
\end{align*}
Then, under the assumption $\tr>0$, the desired conclusion \eqref{azza} readily follows from Lemma \ref{L:uno} and Lemma \ref{L:tre}}. 

{\allowdisplaybreaks
We are thus left with analysing the case  $\tr=0$. Our first observation is that
in view of the crucial Proposition \ref{P:density} there exists a sequence $\{f_n\}\in \So$ such that:
\begin{equation}\label{den}
||f_n-f||_p \ \underset{n\to \infty}{\longrightarrow}\ 0,\ \ \ \ \ \mathscr N^\sA_{\sigma,p}(f_n-f)\ \underset{n\to \infty}{\longrightarrow}\ 0.
\end{equation}
In particular, given $\ve>0$ there exists $n_1(\ve)\in \mathbb N$ such that 
\begin{equation}\label{ullahop}
n\ge n_1(\ve)\ \Longrightarrow\ \frac{4}{p}\left|||f_n||_p^p- ||f||_p^p\right|\leq \frac{\ve}{3}. 
\end{equation}
Now, for every $0<s\leq \sigma$ and  $n\in\mathbb{N}$ we bound
\begin{align}\label{denden}
\left|s \mathscr N^\sA_{s,p}(f)^p - \frac{4}{p} ||f||_p^p\right|&\leq s\left|\mathscr N^\sA_{s,p}(f)^p- \mathscr N^\sA_{s,p}(f_n)^p\right| + \left|s \mathscr N^\sA_{s,p}(f_n)^p - \frac{4}{p} ||f_n||_p^p\right| 
\\
&+ \frac{4}{p}\left|||f_n||_p^p- ||f||_p^p\right|.
\notag
\end{align}
On the other hand, by exploiting \eqref{triangular}, and \eqref{bastasigmauno} in Lemma \ref{L:es}, we obtain
\begin{align*}
&s\left|\mathscr N^\sA_{s,p}(f)^p- \mathscr N^\sA_{s,p}(f_n)^p\right|\leq s\left(\max\left\{\mathscr N^\sA_{s,p}(f),\mathscr N^\sA_{s,p}(f_n)\right\}\right)^{p-1}\left|\mathscr N^\sA_{s,p}(f)- \mathscr N^\sA_{s,p}(f_n)\right|\\
&\leq \left(\max\left\{s^{\frac 1p}\mathscr N^\sA_{s,p}(f),s^{\frac 1p}\mathscr N^\sA_{s,p}(f_n)\right\}\right)^{p-1}s^{\frac 1p}\mathscr N^\sA_{s,p}(f-f_n)\\
&\leq \left(\max\left\{\sigma \mathscr N^\sA_{\sigma,p}(f)^p + \frac{2^{p+1}}{p} ||f||^p_p,\sigma \mathscr N^\sA_{\sigma,p}(f_n)^p + \frac{2^{p+1}}{p} ||f_n||^p_p\right\}\right)^{\frac{p-1}{p}}\\
&\times \left(\sigma \mathscr N^\sA_{\sigma,p}(f-f_n)^p + \frac{2^{p+1}}{p} ||f-f_n||^p_p\right)^{\frac 1p}.
\end{align*}
What is critical here is that the right-hand side of the previous inequality is independent of $s \in (0,\sigma]$, and that in view of \eqref{den} above it converges to $0$ as $n\to\infty$. Hence, there exists $n_2(\ve,\sigma)\in\N$ such that for every $s \in (0,\sigma]$ one has
\begin{equation}\label{uhuh}
n\geq n_2(\ve,\sigma)\ \Longrightarrow\ s\left|\mathscr N^\sA_{s,p}(f)^p- \mathscr N^\sA_{s,p}(f_n)^p\right|\leq \frac{\ve}{3}.
\end{equation}
If we let $n_3(\ve,\sigma)=\max\{n_2(\ve,\sigma), n_1(\ve)\}$, and we fix $\bar n \ge n_3(\ve,\sigma)$, then in view of \eqref{ullahop}, \eqref{denden}  and \eqref{uhuh}, for any $0<s\leq\sigma$ we have
\[
\left|s \mathscr N^\sA_{s,p}(f)^p - \frac{4}{p} ||f||_p^p\right| \le \frac 23 \ve + \left|s \mathscr N^\sA_{s,p}(f_{\bar n})^p - \frac{4}{p} ||f_{\bar n}||_p^p\right|.
\]
At this point we invoke Lemma \ref{L:uno} and Lemma \ref{L:quattro}. Since $f_{\bar n}\in \So$, the combination of these two results allows to conclude that $\underset{s\to 0^+}{\lim} s \mathscr N^\sA_{s,p}(f_{\bar n})^p = \frac{4}{p}||f_{\bar n}||_p^p$.
Therefore, there exists $\bar s = \bar s(\ve,\sigma)< \sigma$ such that   
\begin{equation}\label{veroperS}
0<s<\bar s\ \Longrightarrow\ \left|s \mathscr N^\sA_{s,p}(f_{\bar n})^p - \frac{4}{p} ||f_{\bar n}||_p^p\right|\leq \frac{\ve}{3}.
\end{equation}
Substituting \eqref{veroperS} in the above inequality shows that 
$$0<s<\bar{s}\ \Longrightarrow\ \left|s \mathscr N^\sA_{s,p}(f)^p - \frac{4}{p} ||f||_p^p\right|\leq \ve.
$$ 
This proves the desired conclusion \eqref{azza} also in the case $\tr=0$, thus completing the proof of the theorem.}

\end{proof}


\section{Limiting behaviour of the fractional powers as $s\to 0^+$}\label{S:fracpow}

In this section we analyse the limiting behaviour in $L^p$  of the fractional powers \eqref{defpower} as $s\to 0^+$. In this direction, the main results are Theorem \ref{T:2} and Proposition \ref{P:balap1not} below. 
\begin{theorem}\label{T:2}
Let $1<p<\infty$, and assume \eqref{trace}. If $f\in \underset{0<s<1}{\bigcup} \Bps$, then we have
\begin{equation}\label{yoyo}
\underset{s\to 0^+}{\lim} \As f = f\ \ \ \text{in}\ \Lp.
\end{equation}
When $p = 1$ the limit relation \eqref{yoyo} continues to be valid if $\tr  >0$, but it fails when $\tr = 0$. In such case, in fact, for every nontrivial $f\in\So$, with $f\geq 0$, the $\underset{s\to 0^+}{\lim}\ (-\sA)^s f$ does not exist in $L^1$.
\end{theorem}

Theorem \ref{T:2} highlights the special place of $L^1$ in connection with the limiting behaviour of the fractional powers $\As$. A trivial consequence of the above result is that, when $\tr >0$, if $f\in \underset{0<s<1}{\bigcup} \Bps$, then $||\As f||_1 \underset{s\to 0^+}{\longrightarrow} ||f||_1$. This is somewhat close in spirit to Theorem \ref{T:MSallp}. The following result completes the picture by highlighting the different behaviour of $(-\sA)^s$ in $L^1$ when $\tr=0$.

\begin{proposition}\label{P:balap1not}
Let $\tr = 0$, and consider $f\in \underset{0<s<1}{\bigcup}\mathfrak B^\sA_{s,1}$, such that $f\ge 0$. Then,
\[
\underset{s\to 0^+}{\lim}\ ||(-\sA)^s f||_1 = 2||f||_1.
\]
\end{proposition}

We now turn to the proofs of these two results. Similarly to the proof of Theorem \ref{T:MSallp} in Section \ref{S:lim}, that of Theorem \ref{T:2} will be accomplished in a number of steps. We begin with a lemma that clarifies the connection between the Besov spaces $\mathfrak B^\sA_{s,p}$ and the domains of the fractional powers $\As$ in $L^p$ which we denote as $\mathscr L^{2s,p}$. If $0<s<1$ and $\tr\geq 0$, we know from \cite[Section 4]{GThls} and \cite[Proposition 2.13]{GTiso} that $\mathscr L^{2s,p}$ can be characterized as the closure of the functions in $\So$ with respect to the graph norm of $\As$ in $L^p$. The following lemma, which is taken from \cite[Proposition 3.3]{GTfi}, shows that, whenever $f\in \mathfrak B^\sA_{\sigma,p}$, the function $\As f\in L^p$ for any $0<s<\frac{\sigma}{2}$. We reproduce the proof here in order to keep track of the constants in dependence of $s$. 

\begin{lemma}\label{valdylemma}
Assume \eqref{trace}. For $p>1$ and $0<2s<\sigma <1$ we have 
\begin{equation}\label{bastasigmaunobis}
||\As f||_p\leq \frac{s}{\Gamma(1-s)}\left(\frac{2}{(\sigma -2s) p'}\right)^{\frac{1}{p'}}\mathscr N^\sA_{\sigma,p}(f) + \frac{2}{\Gamma(1-s)}||f||_p.
\end{equation}
In particular, \eqref{bastasigmaunobis} shows that $\mathfrak B^\sA_{\sigma,p}\hookrightarrow \mathscr L^{2s,p}$.  
When $p=1$, for any $0< 2s\leq \sigma<1$ we have 
\begin{equation}\label{bastasigmaunobispuno}
||\As f||_1\leq \frac{s}{\Gamma(1-s)}\mathscr N^\sA_{\sigma,1}(f) + \frac{2}{\Gamma(1-s)}||f||_1.
\end{equation}
In particular, this shows that $\mathfrak B^\sA_{\sigma,1}\hookrightarrow \mathscr L^{2s,1}$.
\end{lemma}
\begin{proof}
Let $p\geq 1$, $0<2s\leq\sigma <1$, and fix $f\in \mathfrak B^\sA_{\sigma,p}$. Keeping \eqref{defpower} in mind, we have
\begin{equation}\label{divbigsmall}
\left\|\As f\right\|_p \leq \frac{s}{\Gamma(1-s)} \left\|\int_0^1 \frac{1}{t^{1+s}} \left(P_t f -  f\right) dt\right\|_p + \frac{s}{\Gamma(1-s)} \left\|\int_1^\infty \frac{1}{t^{1+s}} \left(P_t f -  f\right) dt\right\|_p.
\end{equation}
On one hand, by \eqref{tuttenorme} and \eqref{trace}, we have
\begin{align}\label{big}
&\frac{s}{\Gamma(1-s)} \left\|\int_1^\infty \frac{1}{t^{1+s}} \left(P_t f -  f\right) dt\right\|_p\leq 
\frac{s}{\Gamma(1-s)}\int_1^\infty t^{-1-s}\left\|P_tf - f\right\|_p dt \\
&\leq \frac{s}{\Gamma(1-s)}\int_1^\infty t^{-1-s}\left(||P_tf||_p + ||f||_p\right) dt\leq \frac{2s}{\Gamma(1-s)} ||f||_p \int_1^\infty t^{-1-s} dt = \frac{2}{\Gamma(1-s)} ||f||_p.\nonumber
\end{align}
On the other hand, to estimate the integral on the interval $(0,1)$ in \eqref{divbigsmall} we use the following inequality
$$
\left\|P_tf - f\right\|_p\leq \left(\int_{\RN} P_t\left(|f - f(X)|^p\right)(X) dX\right)^{\frac{1}{p}},
$$
which is a consequence of \eqref{stochcompl} and H\"older's inequality. We now consider the cases $p=1$ and $p>1$ separately. When $p=1$, since $2s\leq\sigma$ we have
\begin{align*}
&\frac{s}{\Gamma(1-s)} \left\|\int_0^1 \frac{1}{t^{1+s}} \left(P_t f -  f\right) dt\right\|_1\\
&\leq \frac{s}{\Gamma(1-s)}\int_0^1 \frac{1}{t^{1+s}}\left\|P_tf - f\right\|_1 dt \leq \frac{s}{\Gamma(1-s)}\int_0^1 \frac{1}{t^{1+\frac{2s}{2}}}\int_{\RN}P_t\left(|f-f(X)|\right)(X)dX dt\\
&\leq \frac{s}{\Gamma(1-s)}\int_0^1 \frac{1}{t^{1+\frac{\sigma}{2}}}\int_{\RN}P_t\left(|f-f(X)|\right)(X)dX dt,
\end{align*}
which implies
\begin{equation}\label{mappingpowerspugu1}
\frac{s}{\Gamma(1-s)} \left\|\int_0^1 \frac{1}{t^{1+s}} \left(P_t f -  f\right) dt\right\|_1\leq \frac{s}{\Gamma(1-s)}\mathscr N^\sA_{\sigma,1}(f).
\end{equation}
Putting together \eqref{divbigsmall}, \eqref{big}, and \eqref{mappingpowerspugu1}, we obtain \eqref{bastasigmaunobispuno}. When $p>1$, we assume $\sigma>2s$ and we deduce from H\"older's inequality
\begin{align*}
&\frac{s}{\Gamma(1-s)} \left\|\int_0^1 \frac{1}{t^{1+s}} \left(P_t f -  f\right) dt\right\|_p\leq \frac{s}{\Gamma(1-s)}\int_0^1 \frac{1}{t^{1+s}}\left\|P_tf - f\right\|_p dt\\
&\leq \frac{s}{\Gamma(1-s)}\int_0^1 \frac{1}{t^{1+s-\frac{\sigma}{2}+\frac{\sigma}{2}}}\left(\int_{\RN} P_t\left(|f - f(X)|^p\right)(X) dX\right)^{\frac{1}{p}} dt\\
&\leq \frac{s}{\Gamma(1-s)}\left( \int_0^1 \frac{1}{t^{1+\left(s-\frac{\sigma}{2}\right)p'}}dt\right)^{\frac{1}{p'}}\left(\int_0^1 \frac{1}{t^{1+\frac{\sigma p}{2}}}\int_{\RN} P_t\left(|f - f(X)|^p\right)(X) dX dt\right)^{\frac{1}{p}},
\end{align*}
which implies
\begin{equation}\label{mappingpowerspugu}
\frac{s}{\Gamma(1-s)} \left\|\int_0^1 \frac{1}{t^{1+s}} \left(P_t f -  f\right) dt\right\|_p\leq \frac{s}{\Gamma(1-s)}\left(\frac{2}{(\sigma-2s) p'}\right)^{\frac{1}{p'}}\mathscr N^\sA_{\sigma,p}(f).
\end{equation}
As before, if we combine \eqref{divbigsmall}, \eqref{big}, and \eqref{mappingpowerspugu}, we conclude the proof of \eqref{bastasigmaunobis}.

\end{proof}
 
The following lemma shows that, when $f$ belongs to $\mathfrak B^\sA_{s,p}$, the small time behaviour of $P_t f$ does not influence the limiting behaviour of $(-\sA)^s$, for any $1\leq p<\infty$.

\begin{lemma}\label{L:balapg}
Let $1\leq p<\infty$ and $\tr \geq 0$. Suppose $f\in \underset{0<s<1}{\bigcup}\mathfrak B^\sA_{s,p}
$. Then,
\[
\underset{s\to 0^+}{\lim}  \frac{s}{\Gamma(1-s)} \int_0^1 \frac{1}{t^{1+s}} \left(P_t f -  f\right) dt\ =\ 0  \quad\mbox{ in }L^p.
\]
\end{lemma}

\begin{proof}
Let $\sigma \in (0,1)$ be such that $f\in \mathfrak B^\sA_{\sigma,p}$, and consider $0<s<\frac{\sigma}{2}$. If  $p=1$, then the conclusion follows by letting $s\to 0^+$ in  \eqref{mappingpowerspugu1}. If instead $p>1$, we use \eqref{mappingpowerspugu}.

\end{proof}

The next two lemmas constitute the core of the proof of Theorem \ref{T:2}.

\begin{lemma}\label{P:balapg}
Let $1\leq p<\infty$ and  assume that $f\in \underset{0<s<1}{\bigcup}\mathfrak B^\sA_{s,p}$. If $\tr > 0$, then
\[
\underset{s\to 0^+}{\lim}\ (-\sA)^s f = f \quad\mbox{ in }L^p.
\]
\end{lemma}
\begin{proof}
As in the proof of Proposition \ref{P:balapoint} we use \eqref{identity} to write
\begin{align*}
& 
(-\sA)^s f - f = - \frac{s}{\Gamma(1-s)} \int_0^\infty \frac{1}{t^{1+s}} \left((P_t f -  f) + (1-e^{-t}) f\right) dt
\\
& = - \frac{s}{\Gamma(1-s)} \int_0^1 \frac{1}{t^{1+s}} \left(P_t f -  f\right) dt - \frac{s}{\Gamma(1-s)}\left(\int_0^1 \frac{1-e^{-t}}{t^{1+s}}dt\right) f
\\
& + \frac{s}{\Gamma(1-s)} \left(\int_1^\infty \frac{e^{-t}}{t^{1+s}}dt\right)f- \frac{s}{\Gamma(1-s)} \int_1^\infty \frac{1}{t^{1+s}} P_t f dt.
\end{align*}
The first term goes to $0$ in $L^p$ thanks to Lemma \ref{L:balapg}. Moreover, it is very easy to see that also the second and the third term converge to $0$ in $L^p$ since $f\in L^p$ and the two integrals $\int_0^1 \frac{1-e^{-t}}{t^{1+s}}dt$ and $\int_1^\infty \frac{e^{-t}}{t^{1+s}}dt$ are bounded above uniformly with respect to $s$ (exactly as in the proof of Proposition \ref{P:balapoint}). The proof is completed if we show that
\begin{equation}\label{grtzinf}
\frac{s}{\Gamma(1-s)} \int_1^{\infty} \frac{1}{t^{1+s}} P_t f dt\underset{s\to 0^+}{\longrightarrow}\ 0  \quad\mbox{ in }L^p, \quad\mbox{ for all }f\in L^p.
\end{equation}
To prove \eqref{grtzinf} we observe that Minkowski's inequality and Proposition \ref{P:ptq} imply
\begin{align*}
& \left\|\frac{s}{\Gamma(1-s)} \int_1^{\infty} \frac{1}{t^{1+s}} P_t f dt \right\|_p \leq  \frac{s}{\Gamma(1-s)} \int_1^\infty \frac{1}{t^{1+s}} ||P_t f||_p dt\\
&\leq \frac{s}{\Gamma(1-s)} C(p)||f||_p \int_1^\infty \frac{e^{-t\frac{\tr}{p}}}{t^{1+s}}dt\leq \frac{s}{\Gamma(1-s)} C(p)||f||_p \int_1^\infty e^{-t\frac{\tr}{p}}dt.
\end{align*}
Since $\tr>0$, the last term vanishes as $s\to 0^+$. This establishes \eqref{grtzinf} concluding the proof.

\end{proof}

\begin{lemma}\label{P:balap0}
Let $1< p<\infty$ and  suppose $f\in \underset{0<s<1}{\bigcup}\mathfrak B^\sA_{s,p}$. If $\tr = 0$, then
\begin{equation}\label{smoothprove}
\underset{s\to 0^+}{\lim}\ (-\sA)^s f = f \quad\mbox{ in }L^p.
\end{equation}
\end{lemma}

\begin{proof}
Let $\sigma\in (0,1)$ be such that $f\in \mathfrak B^\sA_{\sigma,p}$. We proceed as in the proof of Lemma \ref{P:balapg}, using \eqref{identity} and Lemma \ref{L:balapg}. The proof is completed once we establish the analogue of \eqref{grtzinf}. The main difference with Lemma \ref{P:balapg} is that, since we now have $\tr =0$, the decay coming from the term $e^{- t \frac{\operatorname{tr} B}{p}}$ in \eqref{tuttenorme} is now lost. To circumvent this difficulty, we first show that the desired conclusion \eqref{smoothprove} does hold when $f\in\So$, and then use a density argument to extend it to $f\in \mathfrak B^\sA_{\sigma,p}$.
In dealing with $f\in\So$, the advantage is that we can exploit the rate of decay given by the $L^1\to L^p$ ultracontractivity of $P_t$, and by the blowup of $V(t)$ for large $t$. Here, the reader should notice the similarities with the arguments in the proofs of Lemmas \ref{L:tre} and \ref{L:quattro}. 

Let then $f\in \So$. In view of Lemma \ref{L:balapg}, to prove \eqref{smoothprove} for such $f$ it suffices to show that
\begin{equation}\label{eqzinf}
\underset{s\to 0^+}{\lim} \frac{s}{\Gamma(1-s)} \int_1^{\infty} \frac{1}{t^{1+s}} P_t f dt\ =\ 0  \quad\mbox{ in }L^p. 
\end{equation}
Now, Proposition \ref{P:ptq} and \eqref{sqrtbound} imply for $1\le t<\infty$,
\[
||P_t f||_{p} \leq \frac{C(p)}{V(t)^{1-\frac{1}{p}}}||f||_{1} \leq C'(p) \frac{||f||_1}{t^{\frac{1}{2p'}}}.
\]
This gives
\begin{align*}
& \left\|\frac{s}{\Gamma(1-s)} \int_1^{\infty} \frac{1}{t^{1+s}} P_t f dt \right\|_p \leq  \frac{s}{\Gamma(1-s)} \int_1^\infty \frac{1}{t^{1+s}} ||P_t f||_p dt\\
&\leq \frac{s}{\Gamma(1-s)} C'(p) ||f||_1 \int_1^\infty \frac{1}{t^{1+s +\frac{1}{2p'}}}dt = \frac{s}{\Gamma(1-s)} C'(p) ||f||_1 \frac{1}{s +\frac{1}{2p'}}\ \underset{s\to 0^+}{\longrightarrow}\ 0.
\end{align*}
This proves \eqref{eqzinf}, and therefore \eqref{smoothprove}, when $f\in \So$. Returning to $f\in \mathfrak B^\sA_{\sigma,p}$,  by Proposition \ref{P:density} there exists a sequence $\{f_n\}\in \So$ such that $f_n\to f$ in $\mathfrak B^\sA_{\sigma,p}$, i.e., \eqref{den} holds. For any $0<s<\frac{\sigma}{2}$ and $n\in\mathbb{N}$, we now use \eqref{bastasigmaunobis} to estimate
\begin{align*}
&\left\|(-\sA)^s f - f\right\|_p\leq \left\|(-\sA)^s \left( f - f_n\right)\right\|_p + \left\|(-\sA)^s f_n - f_n\right\|_p + \left\|f_n - f\right\|_p\\
&\leq \frac{s}{\Gamma(1-s)}\left(\frac{2}{(\sigma -2s) p'}\right)^{\frac{1}{p'}}\mathscr N^\sA_{\sigma,p}(f - f_n) + \left(\frac{2}{\Gamma(1-s)} + 1\right)||f_n-f||_p + \left\|(-\sA)^s f_n - f_n\right\|_p.
\end{align*}
Given $\ve>0$, the sum of the first two terms in the right-hand side of the latter inequality can be made smaller than $\frac{\ve}{2}$ provided that $n$ is large enough, and this can be done uniformly in $s\in (0,\frac{\sigma}{4}]$. Having fixed such $n$, in view of the validity of \eqref{smoothprove} for functions in $\So$, we can make the remaining term  $\left\|(-\sA)^s f_n - f_n\right\|_p\leq \frac{\ve}{2}$ by choosing $s$ small enough. This completes the proof. 

\end{proof}

When $p=1$ Lemma \ref{P:balap0} fails to be true. We have in fact the following.

\begin{lemma}\label{P:balap1notdiff}
Let $\tr = 0$, and suppose that $f\in \underset{0<s<1}{\bigcup}\mathfrak B^\sA_{s,1}$ with $f\ge 0$. Then,
\[
\underset{s\to 0^+}{\lim}\ ||(-\sA)^s f - f||_1 = ||f||_1.
\]
\end{lemma}

\begin{proof}
Suppose $\sigma\in (0,1)$ is such that $f\in \mathfrak B^\sA_{\sigma,1}$, and that moreover $f\ge 0$. We repeat the initial arguments in the proof of Lemmas \ref{P:balapg} and \ref{P:balap0}. After using Lemma \ref{L:balapg}, we are left with understanding what happens to the term
$$\frac{s}{\Gamma(1-s)} \int_1^{\infty} \frac{1}{t^{1+s}} P_t f dt$$
in the limit as $s\to 0^+$ in the $L^1$-topology. Differently from the previous situations, by \eqref{aggstoch} and the hypothesis $\tr=0$ and $f\geq 0$, we have
\begin{align}\label{normuno}
\left\|\frac{s}{\Gamma(1-s)} \int_1^{\infty} \frac{P_t f}{t^{1+s}} dt\right\|_1=\frac{s}{\Gamma(1-s)} \int_1^{\infty}\int_{\RN}\int_{\RN}\frac{f(Y)}{t^{1+s}}p(X,Y,t) dX dY dt=\frac{||f||_1}{\Gamma(1-s)}.
\end{align}
This implies
\[
\underset{s\to 0^+}{\lim}\ ||(-\sA)^s f - f||_1 = \underset{s\to 0^+}{\lim}\ \left\|\frac{s}{\Gamma(1-s)} \int_1^{\infty} \frac{1}{t^{1+s}} P_t f dt\right\|_1 = \underset{s\to 0^+}{\lim}\ \frac{||f||_1}{\Gamma(1-s)}= ||f||_1.
\]

\end{proof}

We explicitly note the following direct consequence of (i) in Proposition \ref{P:balapoint} and of Lemma \ref{P:balap1notdiff}.

\begin{corollary}\label{C:lastbeach}
Let $\tr =0$. For every nontrivial $f\in\So$, with $f\geq 0$, the
$
\underset{s\to 0^+}{\lim}\ (-\sA)^s f$ in $L^1
$
does not exist.
\end{corollary}

We are now ready to provide the
\begin{proof}[Proof of Theorem \ref{T:2}]
Suppose that $1<p<\infty$ and that \eqref{trace} hold. If $f\in \underset{0<s<1}{\bigcup} \Bps$, then the desired conclusion \eqref{yoyo} follows directly from Lemmas \ref{P:balapg} and \ref{P:balap0}. The same conclusion continues to be true when $p=1$ and $\tr >0$ again by Lemma \ref{P:balapg}. When instead $p=1$ and $\tr = 0$, we can appeal to Corollary \ref{C:lastbeach} to complete the proof.
\end{proof}

We remark that the fact that the fractional powers of a suitable operator approximate the identity in the limit as $s\to 0^+$ is not new in the literature. To the best of our knowledge, in an abstract setting this traces back to Balakrishnan's 1960 seminal paper \cite{Bala}. Using his representation of the fractional powers $A^s$ in terms of the resolvent, in his Lemma 2.4 Balakrishnan proved that, given a closed linear operator $A$ on a Banach space $X$ with domain $D(A)$ and with a resolvent $R(\la, A)$ satisfying $||\la R(\la, A)|| \le M$ for all $\lambda>0$, then the fractional powers $A^s$ are well-defined and the following is true:
\begin{equation}\label{Balacond}
\la R(\la, A)x\to 0 \mbox{ as }\lambda\to 0^+ \mbox{ for some }x\in D(A)\quad\Longrightarrow\quad A^s x\to x\mbox{ as }s\to 0^+,
\end{equation}
where the convergence is in the norm topology of $X$. We emphasise that the hypothesis in \cite{Bala} do not necessarily imply that $A$ be the infinitesimal generator of a semigroup.

Theorem \ref{T:2} above unravels the abstract result \eqref{Balacond} in the setting of the H\"ormander operators \eqref{K0} and their semigroups \eqref{hsg}. On the one hand, it clarifies the crucial role played by the trace of the drift in the concrete context of the Besov spaces $\Bps$. On the other hand, it shows why $p=1$ occupies a special place in the analysis of the limiting behaviour of $\As$. 
Since these aspects are perhaps better known to the semigroup community than to workers in pde's, in what follows we elucidate the abstract condition in \eqref{Balacond} in the context of the operators $\sA$ in \eqref{K0} (under the hypothesis \eqref{trace}). 
Consider the representation of the resolvent in terms of the semigroup $R(\la, \sA)=\int_0^\infty e^{-\la t}P_t dt$, see for this \cite[Lemma 2.10]{GT}, where also the above mentioned assumption in \cite{Bala}, $||\la R(\la, A)|| \le M$ for all $\lambda>0$, was verified. Recalling that $\So$ is a core for the realization of $\sA$ in $L^p$, we fix $f\in\So$. If $\tr>0$, then Proposition \ref{P:ptq} gives for any $p\geq 1$
\begin{align*}
\left\|\la \int_0^\infty e^{-\la t}P_t f dt\right\|_p&\leq C(p) ||f||_p\ \la  \int_0^\infty e^{-\la t}e^{-t\frac{\tr}{p}}dt = C(p) ||f||_p \frac{\la p}{\la p +\tr }\  \underset{\la\to 0^+}{\longrightarrow }\ 0. 
\end{align*}
If instead $\tr=0$, then from Propositions \ref{P:ptq} and  \ref{P:boom} we obtain for any $p>1$, 
\begin{align*}
&\left\|\la \int_0^\infty e^{-\la t}P_t f dt\right\|_p\leq \la \int_0^1 e^{-\la t}||P_t f||_p dt +\la \int_1^\infty e^{-\la t}||P_t f||_p dt\\
&\leq \la ||f||_p \int_0^1 e^{-\la t} dt + \la C(p) ||f||_1 \int_1^\infty \frac{e^{-\la t}}{V(t)^{\frac{1}{p'}}}dt\leq (1-e^{-\la})||f||_p + \la \frac{C(p)}{c_0} ||f||_1 \int_1^\infty t^{-\frac{1}{2p'}}e^{-\la t}dt\\
&\leq (1-e^{-\la})||f||_p + \la^{\frac{1}{2p'}} \frac{C(p)}{c_0} ||f||_1 \G( 1- (2p')^{-1}) \underset{\la\to 0^+}{ \longrightarrow} 0. 
\end{align*}
This shows the validity for functions $f\in\So$ of the sufficient condition in \eqref{Balacond} for any $p\geq 1$ when $\tr>0$,  and for any $p>1$ when $\tr=0$. On the other hand, we cannot expect the sufficient condition in \eqref{Balacond} to hold in the case $p=1$ and $\tr=0$. If in fact $f\geq 0$, from \eqref{aggstoch} we have
$$\left\|\la \int_0^\infty e^{-\la t}P_t f dt\right\|_1 = \la ||f||_1 \int_0^\infty e^{-\la t} dt= ||f||_1 \mbox{ for every }\la>0.$$

In closing, we present the

\begin{proof}[Proof of Proposition \ref{P:balap1not}]
Let $f\in \underset{0<s<1}{\bigcup}\mathfrak B^\sA_{s,1}
$, $f\geq 0$. By Lemma \ref{L:balapg} and the definition of $(-\sA)^s f$ in \eqref{defpower} we see that, in order to prove the proposition, it suffices to show that 
\begin{equation}\label{claimbeauty}
\underset{s\to 0^+}{\lim} \left\|\frac{-s}{\Gamma(1-s)} \int_1^\infty \frac{1}{t^{1+s}} \left(P_t f -  f\right) dt\right\|_1\ =\ 2||f||_1.
\end{equation}
We observe that by \eqref{aggstoch} we find
\begin{align*}
&\left\|\frac{-s}{\Gamma(1-s)} \int_1^\infty \frac{1}{t^{1+s}} \left(P_t f -  f\right) dt\right\|_1\leq\frac{s}{\Gamma(1-s)} \int_1^\infty \frac{1}{t^{1+s}} \left(||P_t f||_1 +  ||f||_1\right) dt\\
&\leq\frac{2s}{\Gamma(1-s)}||f||_1 \int_1^\infty \frac{1}{t^{1+s}}dt=\frac{2||f||_1}{\Gamma(1-s)}.
\end{align*}
This implies
$$
\underset{s\to 0^+}{\limsup} \,\left\|\frac{-s}{\Gamma(1-s)} \int_1^\infty \frac{1}{t^{1+s}} \left(P_t f -  f\right) dt\right\|_1\leq 2||f||_1.
$$
To establish \eqref{claimbeauty} are thus left with showing that
\begin{equation}\label{claiminf2}
\underset{s\to 0^+}{\liminf} \,\left\|\frac{-s}{\Gamma(1-s)} \int_1^\infty \frac{1}{t^{1+s}} \left(P_t f -  f\right) dt\right\|_1\geq 2||f||_1.
\end{equation}
We argue similarly to the proof of \eqref{claiminf} in Lemma \ref{L:quattro}. Fix $\ve>0$ and let $K_\ve\subset \RN$ be a compact set such that
\begin{equation}\label{decideK2}
||f||_{L^1\left(\RN\smallsetminus K_\ve\right)} = \int_{\RN\smallsetminus K_\ve}f(\xi)d\xi \leq \ve.
\end{equation}
Hence, we obtain
\begin{align*}
&\left\|\frac{-s}{\Gamma(1-s)} \int_1^\infty \frac{1}{t^{1+s}} \left(P_t f -  f\right) dt\right\|_1 
\\
&\geq \frac{s}{\Gamma(1-s)}\left\|\int_1^\infty \frac{1}{t^{1+s}}f dt\right\|_{L^1\left(K_\ve\right)} - \frac{s}{\Gamma(1-s)}\left\|\int_1^\infty \frac{1}{t^{1+s}}P_t f dt\right\|_{L^1\left(K_\ve\right)}
\\
&+\frac{s}{\Gamma(1-s)}\left\|\int_1^\infty \frac{1}{t^{1+s}}P_tf dt\right\|_{L^1\left(\RN\smallsetminus K_\ve\right)} - \frac{s}{\Gamma(1-s)}\left\|\int_1^\infty \frac{1}{t^{1+s}}f dt\right\|_{L^1\left(\RN\smallsetminus K_\ve\right)}\\
&=\frac{1}{\Gamma(1-s)}||f ||_{L^1\left(K_\ve\right)} +\frac{s}{\Gamma(1-s)}\left\|\int_1^\infty \frac{1}{t^{1+s}}P_tf dt\right\|_1 
\\
&-\frac{2s}{\Gamma(1-s)}\left\|\int_1^\infty \frac{1}{t^{1+s}}P_tf dt\right\|_{L^1\left(K_\ve\right)} - \frac{1}{\Gamma(1-s)}||f||_{L^1\left(\RN\smallsetminus K_\ve\right)}.
\end{align*}
By \eqref{decideK2}, we have $||f ||_{L^1\left(K_\ve\right)}\geq ||f||_1-\ve$. Moreover, since $f\geq 0$  and $\tr=0$, as in \eqref{normuno} we have $s\left\|\int_1^\infty \frac{1}{t^{1+s}}P_tf dt\right\|_1 = ||f||_1$. Finally, from \eqref{PtKt} and \eqref{sqrtbound} we find
\begin{align*}
&\left\|\int_1^\infty \frac{1}{t^{1+s}}P_tf dt\right\|_{L^1\left(K_\ve\right)}\leq \int_1^\infty\frac{1}{t^{1+s}}\int_{\RN} f(Y) \left(\int_{K_\ve}p(X,Y,t)dX\right)dYdt\\
&\leq c_N |K_\ve| ||f||_1 \int_1^\infty \frac{1}{t^{1+s} V(t)} dt \leq c'_N |K_\ve| ||f||_1 \int_1^\infty \frac{1}{t^{1+s+\frac{1}{2}}} dt = c'_N |K_\ve| \frac{2||f||_1}{2s+1}.
\end{align*}
We thus conclude 
\begin{align*}
\left\|\frac{-s}{\Gamma(1-s)} \int_1^\infty \frac{1}{t^{1+s}} \left(P_t f -  f\right) dt\right\|_1 \geq\frac{2||f||_1}{\Gamma(1-s)} -\frac{2\ve}{\Gamma(1-s)} - \frac{4s}{(2s+1)\Gamma(1-s)} c'_N|K_\ve| ||f||_1,
\end{align*}
which implies
$$
\underset{s\to 0^+}{\liminf} \,\left\|\frac{-s}{\Gamma(1-s)} \int_1^\infty \frac{1}{t^{1+s}} \left(P_t f -  f\right) dt\right\|_1 \geq 2||f||_1 - 2\ve.
$$
Since the choice of $\ve$ is arbitrary, the proof of \eqref{claiminf2} is complete.

\end{proof}

 

\bibliographystyle{amsplain}

\end{document}